\documentclass[11pt,a4paper]{article}

\usepackage[left=2.3cm, top=2.3cm,bottom=2.3cm,right=2.3cm]{geometry}

\usepackage{mathtools,amsmath,amssymb,amsthm,mathrsfs,calc,graphicx,stmaryrd,xcolor,dsfont,tikz,pgfplots,bbm,float,array,epsfig,hyperref,color}
\usetikzlibrary{calc}
\usepackage[numbers,square]{natbib}
\usepackage[british]{babel}
\usepackage{amsfonts}              
\usepackage{enumitem}

\usepackage{manfnt} 
\usepackage{tikz} 

\usepackage{tikz}
\usetikzlibrary{arrows.meta, positioning}

\numberwithin{equation}{section}
\numberwithin{figure}{section}

\allowdisplaybreaks[4]

\newtheorem {theorem}{Theorem}[section]
\newtheorem {proposition}[theorem]{Proposition}
\newtheorem {lemma}[theorem]{Lemma}

{\theoremstyle{definition}
\newtheorem{definition}[theorem]{Definition}

\newtheorem*{convention*}{Convention}

\newtheorem {example}[theorem]{Example}
\newtheorem {remark}[theorem]{Remark}
}

{
\theoremstyle{remark}
}

\newcommand{\Cov}{\operatorname{Cov}}

\newcommand{\ii}{{\rm{i}}}

\DeclareMathOperator{\Id}{Id}

\def\ba{\begin{array}}
\def\ea{\end{array}}
\def\bea{\begin{eqnarray} \label}
\def\eea{\end{eqnarray}}
\def\be{\begin{equation} \label}
\def\ee{\end{equation}}
\def\bit{\begin{itemize}}
\def\eit{\end{itemize}}
\def\ben{\begin{enumerate}}
\def\een{\end{enumerate}}

\def\lan{\langle}
\def\ran{\rangle}

\def\BB{\mathbb{B}}
\def\CC{\mathbb{C}}
\def\E{\mathbb{E}}

\def\N{\mathbb{N}}
\def\P{\mathbb{P}}

\def\R{\mathbb{R}}
\def\RRd1{\mathbb{R}^{d+1}}

\def\ZZ{\mathbb{Z}}

\newcommand{\eee}{{\rm e}}

\newcommand{\ind}{\mathbbm{1}}
\newcommand{\eps}{\varepsilon}
\newcommand{\pos}{\mathop{\mathrm{pos}}\nolimits}

\newcommand{\aff}{\mathop{\mathrm{aff}}\nolimits}
\newcommand{\lin}{\mathop{\mathrm{lin}}\nolimits}
\newcommand{\conv}{\mathop{\mathrm{conv}}\nolimits}

\newcommand{\dd}{{\rm d}}
\newcommand{\bsl}{\backslash}
\DeclareMathOperator{\relint}{relint}

\DeclareMathOperator{\diag}{diag}
\DeclareMathOperator{\rank}{rank}

\begin{document}

\title{\bfseries Solid-Angle Formulas for Simplicial Cones and Random Shifted Sections of Polytopes}

\author{Philipp Schange}

\date{}

\maketitle

\begin{abstract}
We consider simplicial cones $C = \operatorname{pos}(v_1, \ldots, v_n) \subseteq \mathbb{R}^n$ whose generators have a Gram matrix of the form $(\langle v_i, v_j \rangle)_{i,j=1}^n = \operatorname{Id}_n+\sum_{j=1}^p \lambda_j u_j u_j^T$, where $p\in \{0,\ldots, n\}$, $\lambda_1, \ldots, \lambda_p > -1$ and $u_1,\ldots, u_p \in \R^n$ are orthogonal unit vectors. We derive an explicit $p$-dimensional integral formula for the solid angle of $C$ involving the analytic continuation of the standard normal distribution function.

As an application, explicit formulas are obtained for the expected face numbers of sections of acute orthocentric simplices and crosspolytopes with uniformly distributed linear subspaces translated by a fixed vector. The translation vector is assumed to lie outside the affine hulls of all proper faces of the polytope. All solid angles entering these formulas are expressed as integrals of dimension at most two. This application includes regular simplices and regular crosspolytopes as special cases.

\end{abstract}

\tableofcontents

\section{Introduction}

\subsection{Solid angles of simplicial cones}

For linearly independent $v_1,\ldots,v_n\in\R^n$, consider the simplicial cone
$$
C
=
\pos(v_1,\ldots,v_n)
:=
\left\{ \sum_{i=1}^n a_i v_i: a_1\geq 0, \ldots, a_n \geq 0 \right\}.
$$
Let $\lan \,\cdot\, , \, \cdot \, \ran$ and $\|\,\cdot \, \|$ denote the standard Euclidean scalar product and norm, respectively.
The solid angle of $C$ is
$$
\alpha(C)
:=
\frac{\lambda_n(C \cap \BB^n)}{\lambda_n(\BB^n)},
$$
where $\BB^n = \{x\in \R^n: \|x\| \leq 1\}$ is the unit ball in $\R^n$ and $\lambda_n(\,\cdot\,)$ denotes the Lebesgue measure on $\R^n$.
For lower-dimensional cones, the solid angle is understood relative to their linear hull.

In dimension $n=3$, explicit formulas for $\alpha(C)$ date back to~\citet{euler_mensura} and~\citet{lagrange_1799solutions}; see, e.g., \cite{eriksson_solid_angles}. For arbitrary dimension $n$,~\citet{aomoto} obtained an expansion for $\alpha(C)$ in terms of an infinite hypergeometric series in the $\binom n2$ variables $\langle v_i/\|v_i\|, v_j/\|v_j\|\rangle$, $1\leq i < j \leq n$. \citet{ribando} rediscovered Aomoto's series and gave a simplified proof.  See~\cite[Section~3]{beck_robins_sam} for a discussion of these results.
\citet{fitisone_zhou_solid_angle} developed methods to compute solid angles based on Aomoto's series.

As a probabilistic approach, solid angles of simplicial cones can be viewed as orthant probabilities of centered multivariate Gaussian distributions. More precisely, if
$$
\Gamma := \left( \lan v_i, v_j \ran \right)_{i,j=1}^n
$$
is the Gram matrix of the generators of $C=\pos(v_1, \ldots, v_n)$, then $\alpha(C)= \P(Z\in (0,\infty)^n)$, where $Z$ is a centered Gaussian
vector with covariance matrix $\Gamma^{-1}$; see, for example,~\cite[Proposition~5.1]{kabluchko_zaporozhets_gauss_simplex}.
Expressing this probability in terms of the density of $Z$ gives
\begin{equation}\label{eq:intro_general_solid_angle}
\alpha(C)
=
\frac{\sqrt{\det\Gamma}}{(2\pi)^{n/2}}
\int_{(0,\infty)^n} \exp\left( -\frac12 x^T\Gamma x \right)\,\dd x.
\end{equation}
Computing the solid angle hence generally leads to an $n$-dimensional integral.

Here, we simplify the integral representation~\eqref{eq:intro_general_solid_angle} for simplicial cones with generators whose Gram matrix is of the form $D + Q$, where $D$ is a diagonal matrix with positive entries on the diagonal, and $Q \in \R^{n\times n}$ is symmetric and has rank at most $p\in \{0, 1, \ldots, n\}$. As we shall see, this is equivalent to the existence of generators $v_1,\ldots,v_n$ of the cone with Gram matrix
\begin{equation}\label{eq:intro_low_rank_Gram}
\Gamma
=
\Id_n+\sum_{j=1}^p\lambda_j u_ju_j^T,
\end{equation}
where $u_1,\ldots,u_p\in\R^n$ are orthogonal unit vectors, $\lambda_1,\ldots,\lambda_p>-1$, and $\Id_n \in \R^{n\times n}$ denotes the identity matrix. The first main result replaces the $n$-dimensional integral in~\eqref{eq:intro_general_solid_angle} by an integral over $\R^p$.
To formulate this,
set $\ii=\sqrt{-1}$ and
let
$$
\Phi: \CC \rightarrow \CC, \quad \Phi(z) = \frac{1}{2}+\frac{1}{\sqrt{2 \pi}} \int_{0}^{z} \eee^{-x^2/2} \dd x
$$
denote the analytic continuation of the standard normal distribution function.

\begin{theorem}\label{theo:intro_solid_angle}
Let $v_1,\ldots,v_n\in\R^n$ be linearly independent vectors whose Gram
matrix is given by~\eqref{eq:intro_low_rank_Gram}. Then the solid angle
of $C=\pos(v_1,\ldots,v_n)$ is
$$
\alpha(C)
=
\frac{1}{(2\pi)^{p/2}} \int_{\R^p} \eee^{-(y_1^2+ \ldots + y_p^2)/2} \prod_{r=1}^n \Phi \Bigg( \ii \sum_{j=1}^p \frac{\sqrt{\lambda_j}}{\sqrt{1+\lambda_j}} (u_j)_r y_j \Bigg)\, \dd y_1 \cdots \dd y_p,
$$
where we use the convention $\sqrt{x}=\ii\sqrt{-x}$ for $x<0$.
\end{theorem}

Although every simplicial cone admits a representation of the form~\eqref{eq:intro_low_rank_Gram} with $p=n$, the formula reduces the dimension of integration from $n$ to $p$ whenever the generators can be chosen so that $p<n$.

For $p=0$, the solid angle $2^{-n}$ of an orthant is recovered. The case $p=1$ includes the pointed parts of tangent and normal cones of regular simplices considered in~\cite{vershik_sporyshev, kabluchko_zaporozhets_absorption} and, more generally, the subclass of orthocentric cones covered by~\cite[Corollary~7.6]{kabluchko_schange_angles_orthocentric_simplices}.
The case $p=2$ occurs for cones over shifted faces of regular (or acute orthocentric) simplices and symmetric (or asymmetric) crosspolytopes. We use this observation to derive explicit formulas for the expected face numbers of random sections of these polytopes; see Section~\ref{sec:intro_random_sections}.

\subsection{Random sections of acute orthocentric simplices and crosspolytopes}\label{sec:intro_random_sections}
Let $P\subseteq\R^n$ be an $n$-dimensional polytope, and fix $x\in\R^n$. For $\ell\in\{1,\ldots,n-1\}$, let $L$ be uniformly distributed on the Grassmannian of $(n-\ell)$-dimensional linear subspaces of $\R^n$. We consider the random section
$$
P\cap(x+L).
$$
The quantities of interest are the expected face numbers
$$
\E f_j \left( P\cap(x+L) \right),
\qquad
j\in\{0,\ldots,n-\ell-1\},
$$
where $f_j(Q)$ denotes the number of $j$-dimensional faces of a polytope $Q$.
By convention, faces are non-empty.
We set $f_j(\varnothing)=0$ to cover the case of an empty section.

\citet{KabluchkoSeidel2022} obtained explicit formulas for these expected face numbers when $x=0$ and $P$ is a regular simplex, a regular crosspolytope, or a cube, each centered at the origin. Here, we allow fixed shifts $x$ satisfying $x\notin\aff(F)$ for every proper face $F$ of $P$, where $\aff(F)$ denotes the affine hull of $F$.

We consider the following two classes of polytopes $P$ that generalize regular simplices and regular crosspolytopes.

\begin{enumerate}
\item Acute orthocentric simplices with orthocenter at the origin: For \(n\geq2\), these are simplices whose altitudes meet at the origin and whose interior contains the origin. They can be represented as
$$
P
=
\conv(y_0,\ldots,y_n),
$$
where, for some $\tau_0,\ldots,\tau_n>0$ and $s:=\tau_0^2+\cdots+\tau_n^2$, the vertices  $y_0, \ldots, y_n \in \R^n$ satisfy
$$
\lan y_i,y_k\ran
=
-\frac1s+\frac{\delta_{ik}}{\tau_i^2},
\qquad
\text{ for all } i,k\in\{0,\ldots,n\}.
$$
Here, $\conv(\, \cdot \,)$ denotes the convex hull and $\delta_{ik}$ denotes the Kronecker delta.

In particular, all regular simplices centered at the origin are obtained by taking $\tau_0 = \ldots = \tau_n$.

\item Symmetric and asymmetric crosspolytopes: These are the polytopes
$$
P
=
\conv \left( \tau_1^+ e_1, \ldots, \tau_n^+ e_n, -\tau_1^- e_1, \ldots, -\tau_n^- e_n \right),
$$
where $e_1,\ldots,e_n$ is the standard orthonormal basis of $\R^n$ and $\tau_1^+, \ldots, \tau_n^+, \tau_1^- ,\ldots, \tau_n^->0$.

Taking all $2n$ parameters equal yields a regular crosspolytope centered at the origin.
\end{enumerate}

If $P$ is isometric to one of the polytopes above, applying a suitable Euclidean isometry to both $P$ and $x$ reduces the problem to the corresponding representation above. This follows from the rotational invariance of the distribution of $L$. In particular, the formulas cover regular simplices and regular crosspolytopes in arbitrary positions.

As in~\cite{KabluchkoSeidel2022}, the computation of these expected face numbers reduces to determining the probabilities that $L\cap\pos(F-x)\neq\{0\}$, where $F$ ranges over the $(j+\ell)$-dimensional faces of $P$. By the conic Crofton formula, these probabilities can be expressed in terms of the conic intrinsic volumes of
$\pos(F-x)$.

For both classes, every proper face $F$ is a simplex and $x\notin \aff(F)$, so that $\pos(F-x)$ is a simplicial cone. The cones $\operatorname{pos}(F-x)$, their faces, and the pointed parts of the corresponding normal cones admit generators whose Gram matrices have the form~\eqref{eq:intro_low_rank_Gram} with $p\leq 2$ in the respective dimensions. Theorem~\ref{theo:intro_solid_angle} therefore expresses all solid angles entering these conic intrinsic volumes as integrals of dimension at most $2$.
Explicit formulas for the conic intrinsic volumes and the expected face numbers are given in Proposition~\ref{prop:intrinsic_volumes_orthocentric_cone} and Theorem~\ref{theo:expected_faces_shifted_sections}.

\section{The class of cones whose Gram matrix is a small rank perturbation of the identity matrix}\label{sec:class_of_cones}

In this section we study the class of simplicial cones with generators with a Gram matrix of the form~\eqref{eq:intro_low_rank_Gram}. We show that this Gram-matrix structure is preserved under passing to faces and, for full-dimensional cones, to the polar cone.

Let $n\in\N$. A \emph{polyhedral cone} in $\R^n$ is of the form
$$
C=\pos(v_1,\ldots,v_m)
:=
\left\{\sum_{i=1}^m a_i v_i:a_1,\ldots,a_m\geq 0 \right\}
$$
for finitely many vectors $v_1,\ldots,v_m\in\R^n$.
We use the convention $\pos(\varnothing)=\{0\}$. 
We write $\lin(C)$ for the linear hull of $C$ and put $\dim(C):=\dim(\lin(C))$.

The cone is called a \emph{simplicial cone} if it admits a set of linearly independent generators. In particular, $\{0\}$ is regarded as a zero-dimensional simplicial cone.

Recall the class of cones of interest in this section: We consider cones that admit generators $v_1, \ldots, v_n$ whose Gram matrix $\Gamma = (\lan v_i, v_j\ran)_{i,j=1}^n$ has the form
\begin{equation}\label{eq:Gram_matrix_with_orthonormal_perturbation}
\Gamma
=
\Id_n + \sum_{j=1}^p \lambda_j u_j u_j^T,
\end{equation}
with $p\in \{0,1, \ldots, n\}$, $\lambda_1, \ldots, \lambda_p \in \R$, and orthogonal unit vectors $u_1, \ldots, u_p \in \R^n$.

The generators of a cone are not uniquely determined, multiplying each of them with an arbitrary positive number does not change the cone. The following lemma shows that a cone admits generators satisfying~\eqref{eq:Gram_matrix_with_orthonormal_perturbation} if and only if it admits generators whose Gram matrix is sum of a diagonal matrix with positive diagonal entries, and a symmetric matrix of rank at most $p$.

\begin{lemma}\label{lem:characterization_Gram_matrices_with_orthonormal_perturbation}
Let $C \subseteq \R^n$ be an $n$-dimensional simplicial cone. Let $p\in \{0,1,\ldots, n\}$. The following assertions are equivalent.
\begin{enumerate}
\item $C=\pos(v_1, \ldots, v_n)$ for some $v_1, \ldots, v_n\in \R^n$ with Gram matrix
$$
\left( \langle v_i, v_j \rangle \right)_{i,j=1}^n
=
\Id_n + \sum_{j=1}^p \lambda_j u_j u_j^T,
$$
where $\lambda_1, \ldots, \lambda_p \in \R$ and $u_1, \ldots, u_p \in \R^n$ are orthonormal.
\item $C=\pos(\widetilde{v_1}, \ldots, \widetilde{v_n})$ for some $\widetilde{v_1}, \ldots, \widetilde{v_n} \in \R^n$ with Gram matrix
$$
\left( \langle \widetilde{v_i}, \widetilde{v_j} \rangle \right)_{i,j=1}^n
=
\diag(d_1, \ldots, d_n) + Q,
$$
where $d_1, \ldots, d_n>0$ and $Q \in \R^{n\times n}$ is symmetric with $\rank(Q)\leq p$.
\end{enumerate}
\end{lemma}
\begin{proof}
(a) $\Rightarrow$ (b): Given the representation in (a), choose $d_1=\cdots=d_n=1$ and set
$
Q
:=
\sum_{j=1}^p\lambda_j u_j u_j^T.
$
Then $Q$ is symmetric and the image of $Q$ is contained in $\lin(u_1,\ldots,u_p)$, so that $\rank(Q)\leq p$.

\medskip
\noindent
(b) $\Rightarrow$ (a): Given is the representation in (b), that is, $
( \langle \widetilde{v_i}, \widetilde{v_j} \rangle )_{i,j=1}^n
=
D+ Q,
$  with $D:=\diag(d_1,\ldots, d_n)$. Then,
$$
D+Q
=
D^{1/2} (\Id_n + D^{-1/2} Q D^{-1/2}) D^{1/2}.
$$
Setting $v_i = \widetilde{v_i}/\sqrt{d_i}$ for all $i=1,\ldots, n$, hence $C=\pos(v_1, \ldots, v_n)$ and
$$
\left( \langle v_i, v_j \rangle \right)_{i,j=1}^n
=
\Id_n + D^{-1/2} Q D^{-1/2}.
$$
The matrix $D^{-1/2}QD^{-1/2}$ is real and symmetric. By the spectral theorem, it admits a representation
$$
D^{-1/2}QD^{-1/2}
=
\sum_{j=1}^r\lambda_j u_j u_j^T,
$$ where $ r=\rank(D^{-1/2}QD^{-1/2})=\rank(Q)\leq p$, the numbers $\lambda_1,\ldots,\lambda_r$ are non-zero, and $u_1,\ldots,u_r$ are orthonormal. If $r<p$, extend $u_1,\ldots,u_r$ to an orthonormal family $u_1, \ldots, u_r, \ldots, u_p$ and set $\lambda_{r+1} = \ldots = \lambda_p=0$.
\end{proof}

Gram matrices of linearly independent generators are positive definite. Conversely, for every positive definite and symmetric matrix $\Gamma$, there exist linearly independent vectors $v_1, \ldots, v_n \in \R^n$ with Gram matrix $\Gamma$.
The next lemma gives a necessary and sufficient condition for the matrix in~\eqref{eq:Gram_matrix_with_orthonormal_perturbation} to be positive definite, as well as its determinant.

\begin{lemma}\label{lem:equiv_condition_for_pos_def_for_orthonormal_perturb}
Let $p\in \{0,1, \ldots, n\}$, let $\lambda_1, \ldots, \lambda_p \in \R$, and let $u_1, \ldots, u_p \in \R^n$ be orthonormal. Then,
\begin{enumerate}
\item The matrix
$
\Id_n + \sum_{j=1}^p \lambda_j u_j u_j^T
$
is positive definite if and only if $\lambda_1>-1, \ldots, \lambda_p>-1$.
\item Its determinant is given by
$$
\det\left(
\Id_n + \sum_{j=1}^p \lambda_j u_j u_j^T
\right)
=
(1+\lambda_1) \cdots (1+\lambda_p).
$$
\end{enumerate}
\end{lemma}
\begin{proof}
Extend $u_1,\ldots,u_p$ to an orthonormal basis $u_1,\ldots,u_n$ of $\R^n$. We determine the eigenvalues of
$$
\Gamma
:=
\Id_n+\sum_{j=1}^p\lambda_j u_j u_j^T
$$
with respect to this basis. For every $k\in\{1,\ldots,p\}$, orthonormality gives
$$
\Gamma u_k
=
u_k+\sum_{j=1}^p\lambda_j u_j u_j^T u_k
=
u_k+\lambda_k u_k
=
(1+\lambda_k)u_k.
$$
Moreover, for every $k\in\{p+1,\ldots,n\}$, the vector $u_k$ is orthogonal to $\lin(u_1,\ldots,u_p)$, and therefore
$$
\Gamma u_k
=
u_k+\sum_{j=1}^p\lambda_j u_j u_j^T u_k
=
u_k.
$$
Thus the eigenvalues of $\Gamma$,  counted with multiplicity, are $1+\lambda_1,\ldots,1+\lambda_p$ together with $n-p$ additional eigenvalues equal to $1$. Hence $\Gamma$ is positive definite if and only if $1+\lambda_j>0$ for all $j=1,\ldots,p$. Moreover,
$$
\det(\Gamma)
=
(1+\lambda_1) \cdots (1+\lambda_p) \cdot 1^{n-p}
=
(1+\lambda_1) \cdots (1+\lambda_p).
$$
\end{proof}

We next describe the Gram matrices of the faces and the polar cone of a full-dimensional simplicial cone in this class. More precisely: The polar cone, as well as every face of the original cone in its linear hull, again have generators whose Gram matrix is of the form~\eqref{eq:Gram_matrix_with_orthonormal_perturbation}.

The \emph{polar cone} of a polyhedral cone $C \subseteq \R^n$ is defined by
$$
C^{\circ}
:=
\{y \in \R^n: \lan y,x \ran \leq 0 \text{ for all } x\in C\}.
$$
It is again a polyhedral cone, and the double polar theorem says $(C^\circ)^\circ=C$.

\begin{proposition}[Gram matrices of faces and polar cones]
\label{prop:faces_polars_cones_perturbation}
Let $v_1, \ldots, v_n \in \R^n$ be linearly independent with Gram matrix
$$
\left( \langle v_i, v_j \rangle \right)_{i,j=1}^n
=
\Id_n + \sum_{j=1}^p \lambda_j u_j u_j^T,
$$
where $p\in \{0,1,\ldots, n\}$, $\lambda_1, \ldots, \lambda_p \in \R$, and $u_1, \ldots, u_p \in \R^n$ are orthonormal. Put
$$
C
:=
\pos(v_1, \ldots, v_n).
$$
\begin{enumerate}
\item The polar cone of $C$ has the form $C^{\circ} = \pos(w_1, \ldots, w_n)$ for $w_1, \ldots, w_n \in \R^n$ with
\begin{equation}\label{eq:general_repr_polar_cone_Gram_matrix}
\left( \langle w_i, w_j \rangle \right)_{i,j=1}^n
=
\Id_n + \sum_{j=1}^p \frac{-\lambda_j}{1+\lambda_j} u_j u_j^T.	
\end{equation}

\item For every $I\subseteq\{1,\ldots,n\}$, $I \neq \varnothing$, there exist orthonormal vectors $\widetilde{u}_1, \ldots, \widetilde{u}_{\min\{p,|I|\}} \in\R^{|I|}$ and numbers $\widetilde{\lambda}_1, \ldots, \widetilde{\lambda}_{\min\{p,|I|\}} > -1$ such that
\begin{equation}\label{eq:general_repr_face_Gram_matrix}
	( \lan v_i,v_j \ran)_{i,j\in I}
	=
	\Id_{|I|} + \sum_{k=1}^{\min\{p,|I|\}} \widetilde{\lambda}_k \widetilde{u}_k \widetilde{u}_k^T.
\end{equation}
\end{enumerate}
\end{proposition}

It turns out that there are generators $w_1, \ldots, w_n$ of the polar cone $C^{\circ}$ whose Gram matrix is inverse to the Gram matrix of $v_1, \ldots, v_n$. For the proof of Proposition~\ref{prop:faces_polars_cones_perturbation}, let us therefore give the inverse of the matrix~\eqref{eq:Gram_matrix_with_orthonormal_perturbation}.

\begin{lemma}\label{lem:inverse_of_matr_with_orthonormal_perturbation}
	For $p\in \{0,1, \ldots, n\}$, $\lambda_1, \ldots, \lambda_p \in \R$, and  orthonormal $u_1, \ldots, u_p \in \R^n$, denote
	$$
	\Gamma
	=
	\Id_n + \sum_{j=1}^p \lambda_j u_j u_j^T.
	$$
	If $\Gamma$ is non-singular, then the inverse of $\Gamma$ is given by
	$$
	\Gamma^{-1}
	=
	\Id_n + \sum_{j=1}^p \frac{-\lambda_j}{1+\lambda_j} u_j u_j^T.
	$$
\end{lemma}
\begin{proof}
	By Lemma~\ref{lem:equiv_condition_for_pos_def_for_orthonormal_perturb}, $\det(\Gamma) = (1+\lambda_1) \cdots (1+\lambda_p)$. Since $\Gamma$ is non-singular, $1+\lambda_j \neq 0$ for all $j=1,\ldots, p$. Thus, the latter matrix is well-defined.
	A direct computation using orthonormality of $u_1, \ldots, u_p$ shows that
	$$
	\left( \Id_n + \sum_{j=1}^p \lambda_j u_j u_j^T \right)
	\cdot
	\left( \Id_n + \sum_{j=1}^p \frac{-\lambda_j}{1+\lambda_j} u_j u_j^T \right)
	=
	\Id_n,
	$$
	thus these two matrices are inverse to each other.
\end{proof}

\begin{proof}[Proof of Proposition~\ref{prop:faces_polars_cones_perturbation}]
\emph{Proof of (a).} Since $v_1, \ldots, v_n$ are linearly independent, there exist unique vectors $w_1, \ldots, w_n \in \R^n$ such that $\lan v_i, w_j \ran = -\delta_{ij}$ for all $i,j \in \{1,\ldots, n\}$. By conic duality,  $C^{\circ} = \pos(w_1, \ldots, w_n)$.
Let $V \in \R^{n \times n}$ and $W \in \R^{n \times n}$ have columns $v_1, \ldots, v_n$ and $w_1, \ldots, w_n$, respectively. The corresponding Gram matrices are $V^TV$ and $W^TW$. The defining relation for $w_1, \ldots, w_n$ gives $-W^T V = \Id_n$. Since $V$ is non-singular, it follows that $W = (-V^{-1})^T$, and hence
$$
\left( \langle w_i, w_j \rangle \right)_{i,j=1}^n
=
W^T W
=
(-V^{-1}) (-V^{-1})^T
=
(V^T V)^{-1}
=
\left( \left( \langle v_i, v_j \rangle \right)_{i,j=1}^n \right)^{-1}
=
\Gamma^{-1}.
$$
To complete the proof, note that Lemma~\ref{lem:inverse_of_matr_with_orthonormal_perturbation} yields
$$
\Gamma^{-1}
=
\Id_n + \sum_{j=1}^p \frac{-\lambda_j}{1+\lambda_j} u_j u_j^T.
$$

\vspace*{.2cm}
\noindent
\emph{Proof of (b).} Let $I \subseteq \{1, \ldots, n\}$ be non-empty. By writing the original Gram matrix as
$$
\left( \langle v_r, v_s \rangle \right)_{r,s=1}^n
=
\Id_n + \sum_{j=1}^p \lambda_j u_j u_j^T
=
\left( \delta_{rs} + \sum_{j=1}^p \lambda_j (u_j)_r (u_j)_s \right)_{r,s=1}^n,
$$
we can see that the collection $(v_i:i\in I)$ has Gram matrix
$$
\left( \langle v_r, v_s \rangle \right)_{r,s \in I}
=
\Id_{|I|} + \sum_{j=1}^p \lambda_j u_{j,I} u_{j,I}^T,
$$
where $u_{j,I} := ((u_j)_r: r\in I)$.
The matrix $\sum_{j=1}^p \lambda_j u_{j,I} u_{j,I}^T$ is real and symmetric, and has rank at most $\min(p, |I|)$.
Hence, by the spectral theorem, there exist orthonormal vectors
$
\widetilde{u}_1, \ldots, \widetilde{u}_{\min\{p,|I|\}} \in \R^{|I|}
$
and real numbers $\widetilde{\lambda}_1, \ldots, \widetilde{\lambda}_{\min\{p,|I|\}}$ such that
$$
\sum_{j=1}^p \lambda_j u_{j,I} u_{j,I}^T
=
\sum_{j=1}^{\min\{p,|I|\}} \widetilde{\lambda}_j \widetilde{u}_j \widetilde{u}_j^T,
$$
where some of the coefficients $\widetilde \lambda_j$ may be zero.

Finally, since the vectors $(v_i:i\in I)$ are linearly independent, their Gram matrix is positive definite. Hence,
$
\Id_{|I|}
+
\sum_{j=1}^{\min\{p,|I|\}} \widetilde{\lambda}_j \widetilde{u}_j \widetilde{u}_j^T
$
is positive definite, and Lemma~\ref{lem:equiv_condition_for_pos_def_for_orthonormal_perturb} implies that
$
\widetilde{\lambda}_j > -1
$
for all $j=1,\ldots,\min\{p,|I|\}$.
\end{proof}

\begin{remark}
For the polar cone, Proposition~\ref{prop:faces_polars_cones_perturbation} gives explicit parameters in~\eqref{eq:general_repr_polar_cone_Gram_matrix}: The orthonormal vectors $u_j$ remain unchanged, and each $\lambda_j$ is replaced by $-\lambda_j/(1+\lambda_j)$.

For a face of the cone indexed by $I \subseteq \{1, \ldots, n\}$, $I \neq \varnothing$, we can explicitly write the Gram matrix
$$
\left(\langle v_r,v_s\rangle\right)_{r,s\in I}
=
\Id_{|I|}
+
\sum_{j=1}^p \lambda_j u_{j,I}u_{j,I}^T,
$$
with $u_{j,I} := ((u_j)_r: r\in I)$, but these vectors are generally not orthonormal.
Obtaining the representation~\eqref{eq:general_repr_face_Gram_matrix} with orthonormal vectors requires diagonalizing the symmetric matrix
$
\sum_{j=1}^p \lambda_j u_{j,I}u_{j,I}^T.
$
The resulting parameters seem to not generally admit a simple expression in terms of the original cone.
\end{remark}

\section{Solid angle formula}\label{sec:solid_angle_formula}

Let $C\subseteq\R^n$ be a polyhedral cone of dimension $d \geq 1$. Its solid angle is
$$
\alpha(C)
:=
\frac{\lambda_d(C\cap\BB)}{\lambda_d(\BB)}
=
\P[Z\in C],
$$
where $\BB$ is the unit ball in $\lin(C)$ centered at the origin, $\lambda_d$ denotes $d$-dimensional Lebesgue measure on $\lin(C)$, and $Z$ is a standard Gaussian random vector in $\lin(C)$.
By convention $\alpha(\{0\})=1$.


In this section we consider an $n$-dimensional simplicial cone
$
C
=
\pos(v_1, \ldots, v_n)
\subseteq \R^n
$
and denote the Gram matrix of its generators $v_1, \ldots, v_n$ by $\Gamma$. The general formula for $\alpha(C)$ in~\eqref{eq:intro_general_solid_angle} (and Lemma~\ref{lem:solid_angle_simplicial_cone_general_explicit} below) is an $n$-dimensional integral. We shall see that in the case of cones introduced in Section~\ref{sec:class_of_cones}, this dimension reduces.

To state the result, we require the analytic continuation of the cumulative distribution function of the standard normal distribution, which we denote by
\begin{equation}\label{eq:compl_std_normal_distr}
	\Phi: \CC \rightarrow \CC, \quad \Phi(z) = \frac{1}{2}+\frac{1}{\sqrt{2 \pi}} \int_{0}^{z} \eee^{-x^2/2} \dd x.
\end{equation}
The integral is taken along any contour connecting $0$ to $z$. Let $\ii= \sqrt{-1}$.

\begin{theorem}\label{theo:solid_angle_orthonormal_perturb}
Let $n\in \N_0$ and let $v_1, \ldots, v_n \in \R^n$ be linearly independent vectors with Gram matrix
$$
\Gamma
=
\Id_n + \sum_{j=1}^p \lambda_j u_j u_j^T,
$$
where $p\in \{0,1,\ldots, n\}$, $\lambda_1, \ldots, \lambda_p \in \R$, and $u_1, \ldots, u_p\in \R^n$ are orthonormal. Then $\lambda_1, \ldots, \lambda_p>-1$ and the solid angle of the cone $C=\pos(v_1, \ldots, v_n)$ is given by
\begin{equation}\label{eq:solid_angle_with_orthonormal_perturbation}
\alpha(C)
=
\frac{1}{(2\pi)^{p/2}} \int_{\R^p} \eee^{-(y_1^2+ \ldots + y_p^2)/2} \prod_{r=1}^n \Phi \Bigg( \ii \sum_{j=1}^p \frac{\sqrt{\lambda_j}}{\sqrt{1+\lambda_j}} (u_j)_r y_j \Bigg)\, \dd y_1 \cdots \dd y_p.
\end{equation}
\end{theorem}
Here, we use the convention $\sqrt{x}=\ii \sqrt{-x}$ for $x<0$. Empty sums and products are understood to be $0$ and $1$, respectively, and integration over $\R^0$ means evaluation of the integrand at its unique point.

The idea of the proof is as follows. In Lemma~\ref{lem:solid_angle_simplicial_cone_general_explicit} we recall a general $n$-dimensional integral representation for the solid angle of a simplicial cone. After plugging in
$$
x^T \Gamma x
=
\|x\|^2 + \sum_{j=1}^p \lambda_j \lan u_j,x \ran^2,
$$
we can apply the Hubbard--Stratonovich identity~\eqref{eq:Hubbard_Stratonovich} $p$ times in this integral representation. This separates and simplifies the integrals with respect to the original integration variables but at the same time yields $p$ new integration variables.
The integrals with respect to the original variables can be evaluated in terms of $\Phi$ using~\eqref{eq:half_line_Gaussian_integral}. A difficulty lies in justifying the interchange of integrals. Here the orthonormality of $u_1,\ldots,u_p$ and the inequalities $\lambda_j>-1$ are applied.

\begin{lemma}\label{lem:solid_angle_simplicial_cone_general_explicit}
	Let $v_1,\ldots,v_n\in\R^n$ be linearly independent, let $\Gamma = (\lan v_i, v_j \ran)_{i,j=1}^n$ denote their Gram matrix, and put $C=\pos(v_1,\ldots,v_n)$. Then
	\begin{equation}\label{eq:solid_angle_simplicial_cone_general_explicit}
		\alpha(C)
		=
		\frac{\sqrt{\det\Gamma}}{(2\pi)^{n/2}} \int_{(0,\infty)^n} \exp\left(-\frac12 x^T\Gamma x \right) \,\dd x.
	\end{equation}
\end{lemma}

The result is proved in \cite[Proposition~5.1]{kabluchko_zaporozhets_gauss_simplex}. We include a proof for completeness.

\begin{proof}[Proof of Lemma~\ref{lem:solid_angle_simplicial_cone_general_explicit}]
Because of the linear independence of $v_1, \ldots, v_n$, there exist unique vectors $w_1, \ldots, w_n \in \R^n$ such that $\lan v_i, w_j \ran = -\delta_{ij}$ for all $i,j \in \{1, \ldots, n\}$. Then the polar cone of $C=\pos(v_1, \ldots, v_n)$ is given by $C^{\circ} = \pos(w_1, \ldots, w_n)$. As in the proof of Proposition~\ref{prop:faces_polars_cones_perturbation}, the Gram matrix of $w_1, \ldots, w_n$ is given by $(\lan w_i, w_j \ran)_{i,j=1}^n = \Gamma^{-1}$.

Let $Z$ be a standard Gaussian random vector in $\R^n$. By the definition of the solid angle and the double polar theorem,
\begin{align*}
\alpha(C)
&=
\P[Z \in C]
=
\P[\lan Z,w \ran \leq 0 \text{ for all } w\in C^{\circ}]
=
\P[\lan Z, w_1 \ran \leq 0, \ldots, \lan Z, w_n \ran \leq 0]\\
&=
\P[\lan Z, w_1 \ran \geq 0, \ldots, \lan Z, w_n \ran \geq 0].
\end{align*}
The random vector $Y := (\lan Z, w_1\ran, \ldots, \lan Z, w_n \ran)$ is Gaussian distributed with mean zero and covariance matrix
$$
\left( \Cov\left( \lan Z, w_i \ran, \lan Z, w_j \ran \right) \right)_{i,j=1}^n
=
\left( \lan w_i, w_j \ran \right)_{i,j=1}^n
=
\Gamma^{-1}.
$$
Hence $Y$ has density
$$
f_Y(x)
=
\frac{1}{(2\pi)^{n/2} \sqrt{\det(\Gamma^{-1})}} \exp \left(-\frac12 x^T\Gamma x \right)
=
\frac{\sqrt{\det\Gamma}}{(2\pi)^{n/2}} \exp \left(-\frac12 x^T \Gamma x \right),
\qquad
x\in \R^n.
$$
Thus,
$$
\alpha(C)
=
\P[Y\in (0,\infty)^n]
=
\frac{\sqrt{\det \Gamma}}{(2\pi)^{n/2}} \int_{(0,\infty)^n} \exp\left(-\frac12 x^T\Gamma x \right)\,\dd x,
$$
which proves~\eqref{eq:solid_angle_simplicial_cone_general_explicit}.
\end{proof}

\begin{lemma}\label{lem:Gaussian_integral_identities_for_solid_angle}
For every $z\in\CC$ we have
\begin{equation}\label{eq:Hubbard_Stratonovich}
\frac{1}{\sqrt{2\pi}}
\int_{-\infty}^{\infty}
\eee^{-t^2/2} \eee^{tz}
\,\dd t
=
\eee^{z^2/2},
\end{equation}
as well as
\begin{equation}\label{eq:half_line_Gaussian_integral}
\int_0^\infty
\exp\left(-\frac{x^2}2+zx\right)
\,\dd x
=
\sqrt{2\pi}
\eee^{z^2/2} \Phi\left(z\right).
\end{equation}
\end{lemma}

\begin{proof}
For real $z$, completing the square gives
\begin{align*}
\frac{1}{\sqrt{2\pi}} \int_{-\infty}^{\infty} \eee^{-t^2/2+zt} \,\dd t
&=
\frac{\eee^{z^2/2}}{\sqrt{2\pi}} \int_{-\infty}^{\infty} \eee^{-(t-z)^2/2}\,\dd t
=
\eee^{z^2/2},\\
\int_0^\infty \eee^{-x^2/2+zx} \,\dd x
&=
\eee^{z^2/2} \int_{-z}^\infty \eee^{-s^2/2} \, \dd s
=
\sqrt{2\pi} \eee^{z^2/2} \Phi(z),
\end{align*}
where the last equality uses $1-\Phi(-z)=\Phi(z)$.
Since both sides of each identity are entire functions of $z$, the identities extend to every $z\in\CC$ by the identity theorem for analytic functions.
\end{proof}

\begin{proof}[Proof of Theorem~\ref{theo:solid_angle_orthonormal_perturb}]
For $p=0$, we have $\Gamma=\Id_n$, and both sides of~\eqref{eq:solid_angle_with_orthonormal_perturbation} equal $2^{-n}$. This includes $n=0$ by the convention $\alpha(\{0\})=1$. We may therefore assume $p\geq1$.

First note that $\Gamma$ is positive definite as the Gram matrix of linearly independent vectors. Hence, Lemma~\ref{lem:equiv_condition_for_pos_def_for_orthonormal_perturb} yields  $\lambda_1, \ldots, \lambda_p>-1$.
Moreover, for $x\in \R^n$,
$$
x^T \Gamma x
=
x^T x + \sum_{j=1}^p \lambda_j x^T u_j u_j^T x
=
\sum_{r=1}^n x_r^2 + \sum_{j=1}^p \lambda_j \left( \sum_{r=1}^n (u_j)_r x_r \right)^2.
$$
Therefore Lemma~\ref{lem:solid_angle_simplicial_cone_general_explicit} yields
$$
\alpha(C)
=
\frac{\sqrt{\det\Gamma}}{(2\pi)^{n/2}} \int_{(0,\infty)^n} \exp\left(-\frac12 \sum_{r=1}^n x_r^2 \right) \prod_{j=1}^p \exp\Bigg( -\frac{\lambda_j}{2} \left( \sum_{r=1}^n (u_j)_r x_r \right)^2 \Bigg)  \,\dd x.
$$
For every fixed $x\in (0,\infty)^n$, the identity~\eqref{eq:Hubbard_Stratonovich} with $z=\ii \sqrt{\lambda_j} \sum_{r=1}^n (u_j)_r x_r$ gives
$$
\prod_{j=1}^p \exp\Bigg( -\frac{\lambda_j}{2} \left( \sum_{r=1}^n (u_j)_r x_r \right)^2 \Bigg)
=
\frac{1}{(2\pi)^{p/2}} \int_{\R^p} \exp\Bigg( -\frac12 \sum_{j=1}^p t_j^2 + \ii \sum_{j=1}^p \sqrt{\lambda_j} t_j \sum_{r=1}^n (u_j)_r x_r \Bigg) \, \dd t.
$$
Hence,
\begin{equation}\label{eq:pf_solid_angle_orthonormal_perturbation_intermediate}
\alpha(C)
=
\frac{\sqrt{\det\Gamma}}{(2 \pi)^{(n+p)/2}} \int_{(0,\infty)^n} \int_{\R^p} \underbrace{ \exp \Bigg(-\frac12 \sum_{j=1}^p t_j^2 \Bigg) \prod_{r=1}^n \exp \Bigg( -\frac{x_r^2}{2} + \ii x_r \sum_{j=1}^p \sqrt{\lambda_j}(u_j)_r t_j \Bigg)}_{=:f(x,t)} \,\dd t \,\dd x.
\end{equation}
Next, let us justify that Fubini's theorem is applicable, by showing absolute integrability of $f$. By convention for the square root,
\begin{align*}
|f(x,t)|
=
\exp \Bigg( -\frac12 \sum_{r=1}^n x_r^2 - \frac12 \sum_{j=1}^p t_j^2 -\sum_{\lambda_j<0} \sqrt{-\lambda_j} t_j \sum_{r=1}^n (u_j)_r x_r  \Bigg).
\end{align*}
The identity~\eqref{eq:half_line_Gaussian_integral} with
$z=-\sum_{\lambda_j<0}\sqrt{-\lambda_j} (u_j)_r t_j$ yields
\begin{align*}
&\int_{\R^p} \int_{(0,\infty)^n} |f(x,t)| \, \dd x \, \dd t\\
&\qquad =
\int_{\R^p} \exp\Bigg( -\frac12 \sum_{j=1}^p t_j^2 \Bigg) \cdot \prod_{r=1}^n \Bigg( \int_0^{\infty} \exp\Bigg( -\frac{x_r^2}2 -x_r \sum_{\lambda_j<0} \sqrt{-\lambda_j} (u_j)_r t_j \Bigg) \dd x_r \Bigg) \dd t
\\
&\qquad =
(2\pi)^{n/2} \int_{\R^p} \exp \Bigg( -\frac12\sum_{j=1}^p t_j^2 + \frac12\sum_{r=1}^n \bigg( \sum_{\lambda_j<0} \sqrt{-\lambda_j} (u_j)_r t_j \bigg)^2 \Bigg) \prod_{r=1}^n \Phi \bigg( -\sum_{\lambda_j<0} \sqrt{-\lambda_j} (u_j)_r t_j \bigg) \,\dd t.
\end{align*}
Since $u_1,\ldots,u_p$ are orthonormal,
$$
\sum_{r=1}^n \Bigg( \sum_{\lambda_j<0} \sqrt{-\lambda_j} (u_j)_r t_j \Bigg)^2
=
\sum_{\substack{\lambda_j<0\\ \lambda_k<0}} \sqrt{-\lambda_j} \sqrt{-\lambda_k} t_j t_k \sum_{r=1}^n (u_j)_r (u_k)_r
=
\sum_{\lambda_j<0}(-\lambda_j) t_j^2.
$$
Moreover, all arguments of $\Phi$ are real, so $0\leq\Phi\leq1$. It follows that
$$
\int_{\R^p} \int_{(0,\infty)^n} |f(x,t)| \, \dd x \, \dd t
\leq
(2\pi)^{n/2} \int_{\R^p} \exp \Bigg( -\frac12 \sum_{\lambda_j<0} (1+\lambda_j) t_j^2 -\frac12 \sum_{\lambda_j \geq 0} t_j^2 \Bigg) \,\dd t
<
\infty.
$$
This integral is finite since $\lambda_j>-1$ for every $j$.

Thus $f$ is absolutely integrable and Fubini's theorem is applicable in~\eqref{eq:pf_solid_angle_orthonormal_perturbation_intermediate}. It follows that
\begin{align*}
\alpha(C)
=
\frac{\sqrt{\det\Gamma}}{(2\pi)^{(n+p)/2}} \int_{\R^p} \exp \Bigg(-\frac12 \sum_{j=1}^p t_j^2 \Bigg) \cdot \prod_{r=1}^n \Bigg( \int_0^\infty \exp \Bigg( -\frac{x_r^2}{2} +\ii x_r \sum_{j=1}^p \sqrt{\lambda_j} (u_j)_r t_j \Bigg) \,\dd x_r \Bigg) \,\dd t\\
=\frac{\sqrt{\det\Gamma}}{(2\pi)^{p/2}} \int_{\R^p} \exp \Bigg( -\frac12 \sum_{j=1}^p t_j^2 \Bigg) \prod_{r=1}^n \exp \Bigg( -\frac12 \Bigg( \sum_{j=1}^p \sqrt{\lambda_j} (u_j)_r t_j \Bigg)^2 \Bigg)
\cdot \prod_{r=1}^n \Phi \Bigg( \ii\sum_{j=1}^p \sqrt{\lambda_j}(u_j)_r t_j \Bigg) \dd t.
\end{align*}
In the second step, we here applied~\eqref{eq:half_line_Gaussian_integral} with $z= \ii\sum_{j=1}^p\sqrt{\lambda_j}(u_j)_r t_j$ to each of the inner integrals.
Since $u_1,\ldots,u_p$ are orthonormal,
\begin{align*}
\sum_{r=1}^n \Bigg(\sum_{j=1}^p \sqrt{\lambda_j}(u_j)_r t_j \Bigg)^2
&=
\sum_{j,k=1}^p \sqrt{\lambda_j} \sqrt{\lambda_k} t_j t_k \sum_{r=1}^n (u_j)_r (u_k)_r
=
\sum_{j=1}^p \lambda_j t_j^2.
\end{align*}
Therefore,
\begin{align*}
\alpha(C)
=
\frac{\sqrt{\det\Gamma}}{(2\pi)^{p/2}} \int_{\R^p}	\exp \Bigg( -\frac12 \sum_{j=1}^p (1+\lambda_j) t_j^2 \Bigg) \prod_{r=1}^n \Phi\Bigg( \ii \sum_{j=1}^p \sqrt{\lambda_j}(u_j)_r t_j \Bigg) \,\dd t.
\end{align*}
Lemma~\ref{lem:equiv_condition_for_pos_def_for_orthonormal_perturb} says $\det(\Gamma) = (1+\lambda_1) \cdots (1+\lambda_p)$. Substituting $y_j = \sqrt{1+\lambda_j} t_j$ for all $j=1,\ldots, p$ completes the proof.
\end{proof}

\begin{remark}[A probabilistic proof in a special case]
If $\lambda_j\in(-1,0]$ for every $j$, the solid-angle formula~\eqref{eq:solid_angle_with_orthonormal_perturbation} has a more direct probabilistic proof and is a special case of the Gaussian reduction formula in~\cite[Equation~(3)]{Marsaglia63}; see also~\cite[Section~2.3.1]{genz_bretz_book_multivar_normal_t}.
Indeed, as in the proof of Lemma~\ref{lem:solid_angle_simplicial_cone_general_explicit}, we have $\alpha(C)=\P[Z \in (0,\infty)^n]$, where $Z$ is a centered Gaussian vector with covariance matrix
$$
\Gamma^{-1}
=
\Id_n+\sum_{j=1}^p \frac{-\lambda_j}{1+\lambda_j}u_ju_j^T.
$$
Since all coefficients $-\lambda_j/(1+\lambda_j)$ are nonnegative, the matrix $\sum_{j=1}^p \frac{-\lambda_j}{1+\lambda_j}u_ju_j^T$ is positive semidefinite. Consequently, $Z$ has the same distribution as $X+Y$, where $X,Y$ are independent centered Gaussian vectors with covariance matrices $\Id_n$ and $\sum_{j=1}^p \frac{-\lambda_j}{1+\lambda_j} u_ju_j^T$, respectively. Conditioning on $Y$ gives
$$
\alpha(C)
=
\P[X+Y \in (0,\infty)^n]
=
\E\left[\P\left[X_r>-Y_r \text{ for all } r \in \{1,\ldots,n\} \,|\, Y \right]\right]
=
\E\left[\prod_{r=1}^n\Phi(Y_r)\right].
$$
The last equality uses that the $X_r$ are independent standard normal variables, that $X$ is independent of $Y$, and that $1-\Phi(-y)=\Phi(y)$.

Let $G_1,\ldots,G_p$ be independent standard Gaussian random variables. Then $Y$ has the same distribution as
$
\sum_{j=1}^p
\sqrt{\frac{-\lambda_j}{1+\lambda_j}}\,u_jG_j.
$
Putting this expression for $Y$ into the expectation and using that $(G_1,\ldots,G_p)$ and $(-G_1,\ldots,-G_p)$ have the same distribution gives~\eqref{eq:solid_angle_with_orthonormal_perturbation}.
\end{remark}

\begin{example}[Orthocentric cones]\label{ex:orthoc_cones}
Let $n\in\N$, let $\mu_0, \ldots, \mu_n \in \R \bsl \{0\}$, and let $\eps_1,\ldots,\eps_n\in\{\pm1\}$. Suppose that there are linearly independent vectors $\widetilde{v}_1, \ldots, \widetilde{v}_n \in \R^n$ satisfying
$$
\left\langle \widetilde{v}_i, \widetilde{v}_j \right\rangle
=
\frac{1}{\mu_0} + \frac{\delta_{ij}}{\mu_j},
\qquad
i,j \in \{1, \ldots, n\}.
$$
Then, consider the simplicial cone $C=\pos(\eps_1\widetilde v_1,\ldots,\eps_n\widetilde v_n)$. Following~\cite[Definition~4.1]{kabluchko_schange_angles_orthocentric_simplices},
such a cone is called an orthocentric cone.

We now restrict to the subclass with $\mu_1,\ldots,\mu_n>0$. This includes the pointed parts of the tangent and normal cones of acute orthocentric simplices, in particular of regular simplices; see~\cite[Theorem~5.4]{kabluchko_schange_angles_orthocentric_simplices}. Theorem~\ref{theo:solid_angle_orthonormal_perturb} provides the following alternative derivation of the solid-angle formula for this subclass obtained in~\cite[Corollary~7.6]{kabluchko_schange_angles_orthocentric_simplices}.
Indeed, set
$$
\lambda = \frac{\mu_1 + \ldots + \mu_n}{\mu_0},
\qquad
u = \frac{1}{\sqrt{\mu_1 + \ldots + \mu_n}} 
\begin{pmatrix}
\eps_1 \sqrt{\mu_1}\\
\vdots\\
\eps_n \sqrt{\mu_n}
\end{pmatrix}.
$$
Define $v_1:=\eps_1\sqrt{\mu_1}\,\widetilde v_1, \ldots, v_n:=\eps_n\sqrt{\mu_n}\,\widetilde v_n$. Then $C=\pos(v_1,\ldots,v_n)$. Moreover,
\[
\langle v_i,v_j\rangle
=
\delta_{ij}
+\frac{\eps_i\eps_j\sqrt{\mu_i\mu_j}}{\mu_0}
=
\delta_{ij}+\lambda u_i u_j,
\qquad
i,j \in \{1, \ldots, n\}.
\]
Thus the Gram matrix of these generators is
$
\Gamma
:=
(\lan v_i, v_j \ran)_{i,j=1}^n
=
\Id_n+\lambda uu^T.
$
Since the generators are linearly independent, $\Gamma$
is positive definite. Since $\|u\|=1$, Lemma~\ref{lem:equiv_condition_for_pos_def_for_orthonormal_perturb} therefore yields
$$
1+\lambda
=
\frac{\mu_0+\mu_1+\ldots+\mu_n}{\mu_0}>0.
$$
In particular, either $\mu_0>0$ or
$\mu_0<-\mu_1-\cdots-\mu_n$.
Theorem~\ref{theo:solid_angle_orthonormal_perturb} applied to the generators $v_1, \ldots, v_n$ gives
\begin{align*}
	\alpha(C)
	&=
	\frac{1}{\sqrt{2\pi}} \int_{-\infty}^{\infty} \eee^{-y^2/2} \prod_{r=1}^n \Phi\Bigg( \ii\frac{\sqrt{\lambda}}{\sqrt{1+\lambda}} u_r y \Bigg) \,\dd y\\
	&=
	\frac{1}{\sqrt{2\pi}} \int_{-\infty}^{\infty} \eee^{-y^2/2} \prod_{r=1}^n \Phi\Bigg( \eps_r \ii \sqrt{\frac{\mu_r} {\mu_0+\mu_1+\ldots+\mu_n}} \,y \Bigg) \,\dd y.
\end{align*}
In the final equality we used that
$
\frac{\sqrt{\lambda}}{\sqrt{1+\lambda}}u_r
=
\eps_r \sqrt{\frac{\mu_r}{\mu_0+\mu_1+\ldots+\mu_n}}
$
holds both in the case $\mu_0>0$ and in the case $\mu_0<-\mu_1-\ldots-\mu_n$. This recovers the solid-angle formula for this subclass of orthocentric cones given in~\cite[Corollary~7.6]{kabluchko_schange_angles_orthocentric_simplices}.
\end{example}

\section{Expected number of faces in random shifted sections of acute orthocentric simplices and crosspolytopes}

In this section, we apply Theorem~\ref{theo:solid_angle_orthonormal_perturb} to derive explicit formulas for the expected face numbers of random shifted sections of acute orthocentric simplices and asymmetric crosspolytopes.

Let $n\geq 2$, and let $P\subseteq\R^n$ be an $n$-dimensional polytope. Denote the set of its $k$-dimensional faces by $\mathcal F_k(P)$ and put $f_k(P):=|\mathcal F_k(P)|$. By convention, faces are non-empty. We set $f_k(\varnothing)=0$.

Fix $x\in\R^n$ such that $x\notin\aff(F)$ for every proper face $F$ of $P$. For $\ell\in\{1,\ldots,n-1\}$, let $L$ be uniformly distributed on the Grassmannian of $(n-\ell)$-dimensional linear subspaces of $\R^n$. An extension of~\cite[Proposition~4.1]{KabluchkoSeidel2022}, stated in Proposition~\ref{prop:face_number_identity} below, gives
$$
\E f_j(P\cap(x+L))
=
\sum_{F\in\mathcal F_{j+\ell}(P)}
\P((F-x)\cap L\neq\varnothing),
\qquad
j\in\{0,\ldots,n-\ell-1\}.
$$

For both classes of polytopes considered here, every proper face $F$ is a simplex. Together with the assumption on $x$, this ensures that $\pos(F-x)$ is a simplicial cone of dimension $\dim(F)+1$. Moreover, the events
$$
(F-x)\cap L\neq\varnothing
\qquad\text{and}\qquad
L\cap\pos(F-x)\neq\{0\}
$$
are equivalent.

The conic Crofton formula therefore expresses the probabilities above in terms of the conic intrinsic volumes of these cones. These conic intrinsic volumes can be computed from the solid angles of certain cones associated with $\pos(F-x)$. For a suitable choice of generators, the corresponding Gram matrices differ from the identity matrix by a matrix of rank at most two. Theorem~\ref{theo:solid_angle_orthonormal_perturb} then expresses the required solid angles as integrals of dimension at most two.

\subsection{Conic intrinsic volumes}
We first recall facts about conic intrinsic volumes.
For a convex set $A$, write $\relint(A)$ for its relative interior, that is, its interior in its affine hull $\aff(A)$.

\begin{definition}[Normal cone, conic intrinsic volumes]
Let $C \subseteq \R^n$ be a polyhedral cone.
For a face $F$ of $C$, choose $x_0 \in \relint(F)$ and define the normal cone of $C$ at $F$ as
$$
N(F,C)
:=
\left\{ z \in \R^n: \lan z, y-x_0 \ran \leq 0 \text{ for every } y\in C \right\}.
$$
This definition does not depend on the choice of $x_0$. 

The $k$-th conic intrinsic volume of $C$ is
\begin{equation}\label{eq:conic_intr_volume_general}
\upsilon_k(C)
:=
\sum_{F \in \mathcal F_k(C)} \alpha(F) \alpha\left(N(F,C)\right),
\qquad
k\in\{0, \ldots, n \}.
\end{equation}
Here $\mathcal F_k(C)$ denotes the set of $k$-dimensional faces of $C$. Further set $\upsilon_k(C):=0$ for all other $k\in \ZZ$.
\end{definition}

For $\ell \in \{0, \ldots, n\}$, let $L$ be  uniformly distributed on the Grassmannian of all $(n-\ell)$-dimensional subspaces of $\R^n$. The conic Crofton formula says that every polyhedral cone $C \subseteq \R^n$ that is not a linear subspace satisfies
\begin{equation}\label{eq:conic_crofton}
	\P(C \cap L \neq \{0\})
	=
	2(\upsilon_{\ell+1}(C) + \upsilon_{\ell+3}(C) + \ldots).
\end{equation}
For more information on conic intrinsic volumes, we refer to~\cite[Chapter 2.3]{SchneiderConvexCones} and~\cite{AmelunxenLotzDCG17, amelunxen_lotz_mccoy_tropp_living_on_the_edge}.

\subsection{Conic intrinsic volumes of cones over acute orthocentric simplices}\label{sec:cones_over_ortho_simpl}

We now compute the conic intrinsic volumes of simplicial cones in $\R^n$ admitting $d+1$ linearly independent generators whose Gram matrix has the form 
$\Id_{d+1} - tr^T - rt^T$. Here, $d \in \{0, \ldots, n-1\}$, $t \in (0,\infty)^{d+1}$, and $r\in \R^{d+1}$. This form occurs, after rescaling the generators, for the cones over shifted proper faces of acute orthocentric simplices and asymmetric crosspolytopes considered in the next subsection.

The following remark describes a geometric construction of cones with this Gram matrix.

\begin{remark}[Cones over acute orthocentric simplices]
Let $0 \leq d \leq n-1$, let $\tau_0, \ldots, \tau_d>0$ and put $s=\tau_0^2 + \ldots + \tau_d^2$. Consider the simplex $\conv(y_0, \ldots, y_d)$, where the vectors $y_0, \ldots, y_d \in \R^n$ satisfy
\begin{equation*}
\langle y_j, y_k \rangle
=
-\frac1s,
\quad
0 \leq j \neq k \leq d,
\qquad
\|y_j\|^2
=
-\frac1s + \frac1{\tau_j^2},
\quad
0 \leq j \leq d.
\end{equation*}

We shall not need it here, but for $d\geq 2$ these simplices have the following geometric interpretation. For each vertex, its altitude contains $0$, and $0$ lies in the relative interior of the simplex. Conversely, every $d$-dimensional simplex with these two properties satisfies the above identities for suitable parameters $\tau_0,\ldots,\tau_d>0$.

In particular, all regular simplices centered at the origin are obtained by taking $\tau_0= \ldots = \tau_d$.

Now, for $x\in \R^n \bsl \aff(y_0, \ldots, y_d)$ the vectors $y_0-x, \ldots, y_d-x$ are linearly independent, so
$$
\pos(y_0-x, \ldots, y_d-x)
=
\pos(\tau_0(y_0-x), \ldots, \tau_d (y_d-x))
$$
is a simplicial cone. Put
$$
t=(\tau_0,\ldots,\tau_d)^T,
\qquad
r
=
\left( \tau_i \left( \lan y_i,x \ran + \frac{1}{2s} - \frac{\|x\|^2}{2} \right) \right)_{i=0}^d.
$$
The generators
$
\tau_0(y_0-x), \ldots, \tau_d(y_d-x)
$
have the Gram matrix
\begin{align*}
\left( \left\lan \tau_i(y_i-x),\tau_j(y_j-x) \right\ran \right)_{i,j=0}^d
&=
\left( \delta_{ij} - \tau_i \tau_j \left( \lan y_i,x \ran + \lan y_j,x \ran + \frac{1}{s} - \|x\|^2 \right) \right)_{i,j=0}^d\\
&=
\left(\delta_{ij} -t_i r_j -r_i t_j \right)_{i,j=0}^d\\
&=
\Id_{d+1} - tr^T - rt^T.
\end{align*}
The cones over acute orthocentric simplices constructed here need not be orthocentric cones in the sense of Example~\ref{ex:orthoc_cones}.
\end{remark}

To compute the conic intrinsic volumes using~\eqref{eq:conic_intr_volume_general}, we express the required solid angles in terms of the function on the right-hand side of~\eqref{eq:solid_angle_with_orthonormal_perturbation}. We introduce the following notation for this function.

\begin{definition}[The function $\Psi$]
\label{def:solid_angle_rank_p_perturbation}
Let $n \in \N_0$, $p\in \{0,1,\ldots, n\}$, and $\lambda_1, \ldots, \lambda_p >-1$. Let $u_1, \ldots, u_p\in \R^n$ be orthonormal. Define
\begin{equation}\label{eq:def_Psi}
\Psi_n(\lambda_1, \ldots, \lambda_p; u_1, \ldots, u_p)
=
\frac{1}{(2\pi)^{p/2}} \int_{\R^p} \eee^{-(y_1^2+ \ldots + y_p^2)/2} \prod_{r=1}^n \Phi\left( \ii \sum_{j=1}^p \frac{\sqrt{\lambda_j}}{\sqrt{1+\lambda_j}} (u_j)_r y_j \right)\, \dd y_1 \cdots \dd y_p.
\end{equation}
Here, we use the convention $\sqrt{x}=\ii \sqrt{-x}$ for $x<0$.
\end{definition}
In particular, $\Psi_n(;)=2^{-n}$ for every $n\in\N_0$ and $\Psi_0(;)=1$.

By Theorem~\ref{theo:solid_angle_orthonormal_perturb} applied in $\lin(C)$, if a simplicial cone $C$ has generators with Gram matrix
$
\Id_n + \sum_{j=1}^p \lambda_j u_j u_j^T
$, then its solid angle is
$
\alpha(C)
=
\Psi_n(\lambda_1, \ldots, \lambda_p; u_1, \ldots, u_p).
$

The following proposition gives explicit formulas for the solid angles of the faces and normal cones of a cone in the present family, and hence for its conic intrinsic volumes.

\begin{proposition}\label{prop:intrinsic_volumes_orthocentric_cone}
Let $0 \leq d \leq n-1$. Let $t \in (0,\infty)^{d+1}$, $r\in\R^{d+1}$, and let $z_0,\ldots,z_d\in\R^n$ be linearly independent vectors with Gram matrix
\begin{equation}\label{eq:z_i_gram_matrix}
\left(\lan z_i, z_j \ran \right)_{i,j=0}^d
=
\Id_{d+1}-tr^T-rt^T.
\end{equation}
Put
$$
C:=\pos(z_0,\ldots,z_d)\subseteq\R^n.
$$
For $J \subseteq \{0,\ldots,d\}$ write $J^c:=\{0,\ldots,d\} \bsl J$ and $w_J = (w_j)_{j\in J}$ for $w\in \R^{d+1}$. Put $C_J:=\pos(z_j:j\in J)$, with the convention $C_\varnothing=\{0\}$. Then, the conic intrinsic volumes of $C$ are given by
\begin{equation}\label{eq:intrinsic_volumes_orthocentric_cone}
\upsilon_k(C)
=
\sum_{\substack{J \subseteq \{0, \ldots, d\}\\ |J|=k}} \alpha(C_J) \alpha(N(C_J, C)),
\qquad
k\in\{0,\ldots,d+1\},
\end{equation}
where all factors on the right-hand side are given explicitly as follows.

\vspace{.2cm}
\noindent
\underline{Angles of the faces.}

First,
$
\alpha(C_\varnothing)=1
$
by convention.

If $J \neq \varnothing$ and $t_J, r_J$ are linearly independent, then
\begin{equation}\label{eq:face_angle_independent}
\alpha(C_J)
=
\Psi_{|J|} \Bigg(
 \|t_J\|\|r_J\|-\lan t_J,r_J\ran,
-\|t_J\|\|r_J\|-\lan t_J,r_J\ran;
\frac{
\frac{t_J}{\|t_J\|}-\frac{r_J}{\|r_J\|}
}{
\sqrt{ 2 \left( 1 - \frac{\lan t_J, r_J \ran}{\|t_J\| \|r_J\|} \right)}
},
\frac{
\frac{t_J}{\|t_J\|} + \frac{r_J}{\|r_J\|}
}{
\sqrt{ 2 \left( 1 + \frac{\lan t_J, r_J \ran}{\|t_J\| \|r_J\|}\right)}
}
\Bigg).
\end{equation}

If $J \neq \varnothing$ and $t_J,r_J$ are linearly dependent, then
\begin{equation}\label{eq:face_angle_dependent}
\alpha(C_J)
=
\Psi_{|J|} \left( -2 \lan t_J,r_J \ran; \frac{t_J}{\|t_J\|} \right).
\end{equation}

\vspace{.2cm}
\noindent
\underline{Angles of the normal cones.}
\begin{enumerate}[label=\textup{(\roman*)}]
\item First, for $J=\{0,\ldots,d\}$, we have
$
\alpha(N(C,C))=1.
$

\item Suppose that $t$ and $r$ are linearly independent and let $J \subsetneqq \{0,\ldots,d\}$. Put
\begin{align}
\lambda
&=
\|t\| \|r\|- \lan t,r \ran,
&\qquad
\mu
&=
-\|t\| \|r\| - \lan t,r \ran,
\label{eq:prop_intr_vol_parameters_normal_angle_1} \\
u
&=
\frac{
\frac{t}{\|t\|} - \frac{r}{\|r\|}
}{
\sqrt{ 2 \left( 1 - \frac{\lan t,r \ran}{\|t\| \|r\|} \right)}
},
&\qquad
v
&=
\frac{ \frac{t}{\|t\|} + \frac{r}{\|r\|}
}{
\sqrt{ 2 \left( 1+ \frac{\lan t,r \ran}{\|t\| \|r\|} \right)}
}.
\label{eq:prop_intr_vol_parameters_normal_angle_2}
\end{align}
Then $\lambda>0$ and $-1<\mu<0$. Put
\begin{equation}\label{eq:p_I_q_I}
p_J
=
\frac1{\sqrt{2}} \left( \sqrt{ \frac{\lambda}{1+ \lambda}} u_{J^c} + \sqrt{ \frac{-\mu}{1+ \mu}} v_{J^c} \right),
\qquad
q_J
=
\frac1{\sqrt{2}} \left( \sqrt{\frac{\lambda}{1 + \lambda}} u_{J^c} - \sqrt{\frac{-\mu}{1 + \mu}} v_{J^c} \right).
\end{equation}
If $p_J=0$ or $q_J=0$, then
\begin{equation}\label{eq:normal_angle_zero}
\alpha (N(C_J,C) )
=
2^{-|J^c|}.
\end{equation}
If $p_J,q_J$ are linearly independent, then
\begin{multline}\label{eq:normal_angle_independent}
\alpha (N(C_J,C) )
=
\Psi_{|J^c|} \Bigg(
\|p_J\| \|q_J\| - \lan p_J, q_J \ran,
- \|p_J\| \|q_J\| - \lan p_J, q_J \ran;\\
\frac{
\frac{p_J}{\|p_J\|} - \frac{q_J}{\|q_J\|}
}{
\sqrt{ 2 \left( 1- \frac{\lan p_J, q_J \ran}{\|p_J\| \|q_J\|} \right)}
},
\frac{
\frac{p_J}{\|p_J\|} + \dfrac{q_J}{\|q_J\|}
}{
\sqrt{ 2 \left(1 + \frac{\lan p_J, q_J \ran}{\|p_J\| \|q_J\|} \right)}
}
\Bigg).
\end{multline}
If $p_J,q_J$ are non-zero and linearly dependent, then
\begin{equation}\label{eq:normal_angle_dependent}
\alpha (N(C_J,C) )
=
\Psi_{|J^c|} \left( -2 \lan p_J, q_J \ran; \frac{p_J}{\|p_J\|} \right).
\end{equation}

\item Suppose that $t$ and $r$ are linearly dependent. For every proper subset $J \subsetneqq \{0,\ldots,d\}$,
\begin{equation}\label{eq:normal_angle_rank_one}
\alpha ( N(C_J,C)  )
=
\Psi_{|J^c|} \left( \frac{2 \lan t,r \ran}{1 - 2 \lan t,r \ran} \frac{\|t_{J^c}\|^2}{\|t\|^2}; \frac{t_{J^c}}{\|t_{J^c}\|} \right).
\end{equation}
\end{enumerate}
\end{proposition}

The following lemma shows how to write the matrix $-tr^T-rt^T$ in a form needed to apply Theorem~\ref{theo:solid_angle_orthonormal_perturb}.

\begin{lemma}\label{lem:orthonormal_representation_from_t_r}
For $m\in \N$ let $t,r\in \R^m$ and suppose $t\neq 0$.
\begin{enumerate}
\item If $t$ and $r$ are linearly independent, then put
\begin{align*}
u
&:=
\frac{\frac{t}{\|t\|} - \frac{r}{\|r\|}
}{
\sqrt{2 \left( 1- \frac{\lan t,r \ran}{\|t\| \|r\|} \right)}
},
&\qquad
v
&:=
\frac{\frac{t}{\|t\|} + \frac{r}{\|r\|}
}{
	\sqrt{2 \left( 1+ \frac{\lan t,r \ran}{\|t\| \|r\|} \right)}
},\\
\lambda
&:=
\|t\| \|r\| - \lan t,r \ran >0,
&\qquad
\mu
&:=
-\|t\| \|r\| - \lan t,r \ran <0.
\end{align*}
Then the vectors $u$ and $v$ are orthonormal and
$$
-tr^T -rt^T
=
\lambda u u^T + \mu v v^T.
$$
\item If $t$ and $r$ are not linearly independent, then
$$
-tr^T-rt^T
=
-2\lan t,r \ran \left( \frac{t}{\|t\|} \right) \left( \frac{t}{\|t\|} \right)^T.
$$
\end{enumerate}
\end{lemma}
\begin{proof}
\emph{Proof of (a).} Suppose $t, r$ are linearly independent. Hence, the Cauchy-Schwarz inequality entails that $\lambda>0$, $\mu<0$, and that the denominators in the definition of $u$ and $v$ are non-zero.
A direct computation yields $\|u\|^2 = \|v\|^2 = 1$ and $\lan u,v\ran = 0$. Finally,
\begin{align*}
&\lambda u u^T + \mu v v^T\\
&=
\left( \|t\| \|r\| - \lan t,r \ran \right)
\frac{ \left( \frac{t}{\|t\|} - \frac{r}{\|r\|} \right) \left( \frac{t}{\|t\|} - \frac{r}{\|r\|} \right)^T
}{
2 \left( 1- \frac{\lan t,r \ran}{\|t\| \|r\|} \right)
}
+
\left(- \|t\| \|r\| - \lan t,r \ran \right)
\frac{ \left( \frac{t}{\|t\|} + \frac{r}{\|r\|} \right) \left( \frac{t}{\|t\|} + \frac{r}{\|r\|} \right)^T
}{
2 \left( 1+ \frac{\lan t,r \ran}{\|t\| \|r\|} \right)
}\\
&=
\frac{\|t\| \|r\|}{2} \left( \left( \frac{t}{\|t\|} - \frac{r}{\|r\|} \right) \left( \frac{t}{\|t\|} - \frac{r}{\|r\|} \right)^T - \left( \frac{t}{\|t\|} + \frac{r}{\|r\|} \right) \left( \frac{t}{\|t\|} + \frac{r}{\|r\|} \right)^T \right)\\
&=
\frac{\|t\| \|r\|}{2} \left( -2 \left( \frac{t}{\|t\|} \right) \left( \frac{r}{\|r\|} \right)^T -2 \left( \frac{r}{\|r\|} \right) \left( \frac{t}{\|t\|} \right)^T \right)\\
&=
-tr^T-rt^T.
\end{align*}

\noindent
\emph{Proof of (b).}
Suppose $t,r$ are not linearly independent. Since $t\neq 0$, hence $r=ct$ for some $c\in \R$. Then $\lan t,r\ran = c \|t\|^2$ and consequently
$$
-tr^T-r t^T
=
-t (ct)^T - ct t^T
=
-2ctt^T
=
-2\frac{\lan t,r \ran}{\|t\|^2} t t^T
=
-2 \lan t,r \ran \left( \frac{t}{\|t\|} \right) \left( \frac{t}{\|t\|} \right)^T.
$$
\end{proof}

\begin{proof}[Proof of Proposition~\ref{prop:intrinsic_volumes_orthocentric_cone}]

The generating vectors $z_0, \ldots, z_d$ of $C$ are linearly independent. It follows that $C$ is a simplicial cone, with faces
$$
C_J = \pos(z_j: j\in J),
\qquad
J\subseteq\{0,\ldots,d\},
$$
satisfying $\dim(C_J)=|J|$.
Consequently, the definition~\eqref{eq:conic_intr_volume_general} of conic intrinsic volumes immediately proves~\eqref{eq:intrinsic_volumes_orthocentric_cone}. It remains to show the stated formulas for the angles of the faces and their normal cones.

For $J \neq \varnothing$ let $G_J$ be the Gram matrix of $(z_j)_{j\in J}$. By~\eqref{eq:z_i_gram_matrix},
\begin{equation}\label{eq:Gram_matrix_face_C_I}
G_J
=
\Id_{|J|}- t_J r_J^T - r_J t_J^T.
\end{equation}
The matrix $G_J$ is positive definite as the Gram matrix of linearly independent vectors.

\medskip
\noindent
\underline{Angles of the faces.}
Suppose first that $J \neq \varnothing$ and $t_J$ and $r_J$ are linearly independent.
Lemma~\ref{lem:orthonormal_representation_from_t_r}, applied to
$t_J$ and $r_J$, yields
$$
G_J
=
\Id_{|J|} + \lambda_J \widetilde{u}_J \widetilde{u}_J^T + \mu_J \widetilde{v}_J \widetilde{v}_J^T,
$$
where $\lambda_J, \mu_J, \widetilde{u}_J, \widetilde{v}_J$ are the parameters appearing in~\eqref{eq:face_angle_independent}. Since $G_J$ is positive definite, Theorem~\ref{theo:solid_angle_orthonormal_perturb} now gives~\eqref{eq:face_angle_independent}.

\medskip
\noindent
Suppose $J \neq \varnothing$ and that $t_J$ and $r_J$ are linearly dependent. The assumption $t \in (0, \infty)^{d+1}$ yields $t_J \neq 0$. Hence Lemma~\ref{lem:orthonormal_representation_from_t_r} gives
$$
G_J
=
\Id_{|J|} - 2 \lan t_J,r_J \ran \left(\frac{t_J}{\|t_J\|} \right) \left( \frac{t_J}{\|t_J\|} \right)^T.
$$
Since $G_J$ is positive definite, Theorem~\ref{theo:solid_angle_orthonormal_perturb} gives~\eqref{eq:face_angle_dependent}. 

\medskip
\noindent
\underline{Angles of the normal cones.}

\medskip
\noindent
\emph{Proof of (i).} The normal cone of $C$ at itself is $N(C,C)=\lin(C)^{\perp}$, which is a linear subspace. Thus $\alpha(N(C,C)) = 1$.

\medskip
\noindent
\emph{Proof of (ii) and (iii).}
Put $G:=G_{\{0,\ldots,d\}}$. Since $z_0, \ldots, z_d$ are linearly independent, there exist  unique vectors $z_0^*, \ldots, z_d^*\in \lin(C)$ satisfying
$$
\lan z_i^*, z_j \ran
=
-\delta_{ij},
\qquad
\text{ for all }
i,j \in \{0,\ldots,d\}.
$$

For every $j\in\{0,\ldots,d\}$, the vector $z_j^*$ is an outer normal vector, relative to $\lin(C)$, of the facet
$
C_{\{0,\ldots,d\} \bsl \{j\}},
$
since it is orthogonal to this facet and $\lan z_j^*,z_j\ran= -1 < 0$.

For every $J \subsetneqq \{0,\ldots,d\}$, the facets of $C$ containing $C_J$ are precisely all
$
C_{\{0,\ldots,d\} \bsl \{j\}}
$
with $j\in J^c$. By conic duality, the normal cone is the orthogonal direct sum of the positive hull of the corresponding relative outer normal vectors and $\lin(C)^\perp$. Hence,
\begin{equation*}
N(C_J,C)
=
\pos(z_j^*:j\in J^c) \oplus \lin(C)^\perp.
\end{equation*}
Taking the direct orthogonal sum with a subspace $\lin(C)^\perp$ does not change the solid angle. Hence
\begin{equation}\label{eq:normal_angle_polar_cone}
\alpha(N(C_J,C))
=
\alpha( \pos(z_j^*: j \in J^c) ).
\end{equation}
Since $z_0, \ldots, z_d$ have Gram matrix $G$, it follows, as in the proof of Proposition~\ref{prop:faces_polars_cones_perturbation}, that the Gram matrix of $z_0^*,\ldots,z_d^*$ is $G^{-1}$.
Hence, the Gram matrix of the generators in \eqref{eq:normal_angle_polar_cone} is the submatrix
$$
\left( \lan z_i^*, z_j^* \ran \right)_{i,j\in J^c}
=
\left( (G^{-1})_{i,j} \right)_{i,j \in J^c}
=:
(G^{-1})_{J^c, J^c}.
$$

\medskip
\noindent
\emph{Proof of (ii)}. Suppose that $t$ and $r$ are linearly independent, and fix $J \subsetneqq \{0, \ldots, d\}$.
First, Lemma~\ref{lem:orthonormal_representation_from_t_r} gives
$$
G
=
\Id_{d+1}+\lambda uu^T+\mu vv^T,
$$
where the parameters $\lambda, \mu \in \R$ and the orthonormal vectors $u,v \in \R^{d+1}$ are given by~\eqref{eq:prop_intr_vol_parameters_normal_angle_1} and~\eqref{eq:prop_intr_vol_parameters_normal_angle_2}. 
Moreover, the lemma gives
$
\lambda>0,
$
and
$
\mu<0.
$
Since $G$ is positive definite and $u,v$ are orthonormal, Lemma~\ref{lem:equiv_condition_for_pos_def_for_orthonormal_perturb} also yields $\lambda,\mu>-1$. Consequently, $\lambda>0$ and $-1<\mu<0$.
In particular, the square roots appearing in~\eqref{eq:p_I_q_I} are real, so $p_J$ and $q_J$ are well-defined.

Lemma~\ref{lem:inverse_of_matr_with_orthonormal_perturbation} entails
$$
G^{-1}
=
\Id_{d+1} - \frac{\lambda}{1 +\lambda} uu^T - \frac{\mu}{ 1+ \mu} vv^T.
$$
The definition~\eqref{eq:p_I_q_I} of $p_J$ and $q_J$ yields
\begin{align*}
p_Jq_J^T+q_Jp_J^T
&=
\frac{\lambda}{1+\lambda}
u_{J^c}u_{J^c}^T
+
\frac{\mu}{1+\mu}
v_{J^c}v_{J^c}^T.
\end{align*}
Therefore
\begin{equation}\label{eq:normal_Gram_p_I_q_I}
\left( \lan z_i^*, z_j^* \ran \right)_{i,j\in J^c}
=
(G^{-1})_{J^c,J^c}
=
\Id_{|J^c|}- p_J q_J^T - q_J p_J^T.
\end{equation}
We now compute the solid angle in~\eqref{eq:normal_angle_polar_cone} from this Gram matrix, distinguishing the three possible cases for $p_J$ and $q_J$.

If $p_J=0$ or $q_J=0$, then~\eqref{eq:normal_Gram_p_I_q_I} shows that the vectors $(z_j^*:j\in J^c)$ are orthonormal. Hence  $\pos(z_i^*: i\in J^c)$ is a $|J^c|$-dimensional orthant. Then~\eqref{eq:normal_angle_polar_cone} becomes
$
\alpha (N(C_J,C) )= 2^{-|J^c|},
$
which proves~\eqref{eq:normal_angle_zero}.

If $p_J$ and $q_J$ are linearly independent, we apply Lemma~\ref{lem:orthonormal_representation_from_t_r} (a) to the right-hand side of~\eqref{eq:normal_Gram_p_I_q_I}. The resulting parameters and orthonormal vectors are those appearing in~\eqref{eq:normal_angle_independent}. Since $(G^{-1})_{J^c,J^c}$ is positive definite, Theorem~\ref{theo:solid_angle_orthonormal_perturb}, together with~\eqref{eq:normal_angle_polar_cone}, proves~\eqref{eq:normal_angle_independent}.

Finally, if $p_J$ and $q_J$ are non-zero and linearly dependent, then Lemma~\ref{lem:orthonormal_representation_from_t_r} (b) gives
$$
(G^{-1})_{J^c,J^c}
=
\Id_{|J^c|} -2 \lan p_J,q_J \ran \left( \frac{p_J}{\|p_J\|} \right) \left( \frac{p_J}{\|p_J\|} \right)^T.
$$
Applying Theorem~\ref{theo:solid_angle_orthonormal_perturb} and using~\eqref{eq:normal_angle_polar_cone} proves~\eqref{eq:normal_angle_dependent}.

\medskip
\noindent
\emph{Proof of (iii).} Suppose that $t$ and $r$ are linearly dependent.
Lemma~\ref{lem:orthonormal_representation_from_t_r} (b) gives
$$
G
=
\Id_{d+1} -2 \lan t,r \ran \left( \frac{t}{\|t\|} \right) \left( \frac{t}{\|t\|} \right)^T.
$$
Since $G$ is positive definite,  Lemma~\ref{lem:inverse_of_matr_with_orthonormal_perturbation} entails
$$
G^{-1}
=
\Id_{d+1}
+
\frac{2 \lan t,r \ran}{1- 2 \lan t,r\ran} \frac{tt^T}{\|t\|^2}.
$$

Since $J\subsetneqq\{0,\ldots,d\}$ and $t\in(0,\infty)^{d+1}$, we have 
$t_{J^c}\neq 0$ and hence
\begin{align*}
(G^{-1})_{J^c,J^c}
=
\Id_{|J^c|}
+
\frac{2\lan t,r\ran}{1-2\lan t,r\ran}
\frac{t_{J^c}t_{J^c}^T}{\|t\|^2}
=
\Id_{|J^c|}
+
\left(
\frac{2\lan t,r\ran}{1-2\lan t,r\ran}
\frac{\|t_{J^c}\|^2}{\|t\|^2}
\right)
\left(\frac{t_{J^c}}{\|t_{J^c}\|}\right)
\left(\frac{t_{J^c}}{\|t_{J^c}\|}\right)^T.
\end{align*}
Hence Theorem~\ref{theo:solid_angle_orthonormal_perturb}, applied to the cone $\pos(z_i^*:i\in J^c)$ with the Gram matrix computed above, gives the right-hand side of~\eqref{eq:normal_angle_rank_one}. In view of~\eqref{eq:normal_angle_polar_cone}, this proves \eqref{eq:normal_angle_rank_one}.
\end{proof}

\subsection{Random shifted sections of acute orthocentric simplices and asymmetric crosspolytopes}
We now apply Proposition~\ref{prop:intrinsic_volumes_orthocentric_cone} to compute expected face numbers of random shifted sections. We begin with a face-number identity for random linear sections.

\begin{proposition}\label{prop:face_number_identity}
Let $n \geq 2$ and let $Q \subseteq \R^n$ be a full-dimensional polytope such that $0 \notin \aff(F)$ for every proper face $F$ of $Q$. Let $\ell \in \{1,\ldots,n-1\}$, and let $L$ be uniformly distributed on the Grassmannian of $(n-\ell)$-dimensional linear subspaces of $\R^n$. Then, almost surely, for every $j \in \{0,\ldots,n-\ell-1\}$,
\begin{equation}\label{eq:face_number_almost_sure_identity}
f_j(Q \cap L)
=
\sum_{F \in \mathcal F_{j+\ell}(Q)}
\ind_{\{F \cap L \neq \varnothing\}}.
\end{equation}
Here, we use the convention $f_j(\varnothing)=0$.
\end{proposition}
Under the stronger assumption that $0$ lies in the interior of $Q$, this is~\cite[Proposition~4.1]{KabluchkoSeidel2022}. The proof given in~\cite[Section~5.3]{KabluchkoSeidel2022} can be adapted to the assumptions above. In fact, the identity remains valid under the weaker assumption $0\notin\partial Q$, see Proposition~\ref{prop:app_face_number_identity} in the appendix.

We retain the affine-hull condition for the applications below. Whenever a proper face $F$ is a simplex, this condition ensures that its vertices are linearly independent, so that $\pos(F)$ is a simplicial cone of dimension $\dim(F)+1$. Once the corresponding Gram matrices have been identified, we can apply Proposition~\ref{prop:intrinsic_volumes_orthocentric_cone} to the cones over shifted faces of acute orthocentric simplices and asymmetric crosspolytopes. The following theorem gives the resulting formulas.

\begin{theorem}\label{theo:expected_faces_shifted_sections}
Let $n\geq2$, let $\ell\in\{1,\ldots,n-1\}$, and let
$j\in\{0,\ldots,n-\ell-1\}$. Let $L$ be a uniformly distributed
$(n-\ell)$-dimensional linear subspace of $\R^n$.

\begin{enumerate}
\item Let $\tau_0,\ldots,\tau_n>0$, put
$
s=\tau_0^2+\ldots+\tau_n^2,
$
and let
$
P=\conv(y_0,\ldots,y_n)\subseteq\R^n
$
be an $n$-dimensional simplex satisfying
\begin{equation}\label{eq:global_orthocentric_simplex}
\lan y_i,y_k \ran
=
-\frac1s + \frac{\delta_{ik}}{\tau_i^2},
\qquad
i,k \in \{0,\ldots,n\}.
\end{equation}
Then, for every $x \in \R^n$ such that $x \notin \aff(F)$ for every proper face $F$ of $P$, we have
\begin{equation}\label{eq:expected_faces_orthocentric_simplex}
\E f_j(P\cap(x+L))
=
2 \sum_{\substack{I \subseteq \{0,\ldots,n\} \\ |I|=j+\ell+1}}
\sum_{q=0}^{\lfloor j/2 \rfloor} \upsilon_{\ell+1+2q} (\pos(y_i-x: i \in I)).
\end{equation}
Here, for every $I$, the conic intrinsic volumes are given explicitly by
Proposition~\ref{prop:intrinsic_volumes_orthocentric_cone}, with
\begin{equation}\label{eq:parameters_orthocentric_face}
t=(\tau_i)_{i\in I},
\qquad
r=
\left( \tau_i \left( \lan y_i,x \ran + \frac{1}{2s} - \frac{\|x\|^2}{2} \right) \right)_{i \in I}.
\end{equation}

\item Denote the standard orthonormal basis of $\R^n$ by $e_1, \ldots, e_n$. Let
$$
P
=
\conv (\tau_1^+ e_1, \ldots, \tau_n^+ e_n, -\tau_1^- e_1, \ldots, -\tau_n^- e_n),
$$
where $\tau_i^+, \tau_i^->0$ for all $i\in \{1,\ldots, n\}$.

For every $x=(x_1, \ldots, x_n) \in \R^n$  satisfying $x \notin \aff(F)$ for  every proper face $F$ of $P$, we have
\begin{equation}\label{eq:expected_faces_asymmetric_crosspolytope}
\E f_j (P \cap (x+L) )
=
2 \sum_{\substack{ I \subseteq \{1,\ldots,n\} \\ |I|=j+\ell+1}} \sum_{\eps \in \{-1, +1\}^{I}} \sum_{q=0}^{\lfloor j/2 \rfloor}
\upsilon_{\ell+1+2q} \left( \pos \left( \eps_i \tau_i^{[\eps_i]} e_i - x: i\in I \right) \right),
\end{equation}
with the notation
$\tau_i^{[+1]}:=\tau_i^+$
and
$\tau_i^{[-1]}:=\tau_i^-.$
Here, for every $I$ and $\eps$, the conic
intrinsic volumes are given explicitly by
Proposition~\ref{prop:intrinsic_volumes_orthocentric_cone}, with
\begin{equation}\label{eq:parameters_crosspolytope_face}
t
=
\bigg( \frac{1}{\tau_i^{[\eps_i]}} \bigg)_{i\in I},
\qquad
r
=
\bigg( \eps_i x_i - \frac{\|x\|^2}{2 \tau_i^{[\eps_i]}} \bigg)_{i\in I}.
\end{equation}
\end{enumerate}
In both parts, Proposition~\ref{prop:intrinsic_volumes_orthocentric_cone} is applied with $d=j+\ell$ and with the generators indexed by $I$ in place of $\{0,\ldots,d\}$. Thus, the faces of each such cone are indexed by subsets $J\subseteq I$, and the complement $J^c$ in that proposition is understood as $I\bsl J$.
\end{theorem}

\begin{proof}
Let $P$ and $x$ be as in (a) or (b). Applying~\eqref{eq:face_number_almost_sure_identity} to $Q=P-x$ gives
$$
\E f_j(P\cap(x+L))
=
\E f_j((P-x) \cap L)
=
\sum_{F' \in \mathcal F_{j+\ell}(P-x)} \P(F' \cap L \neq \varnothing )
=
\sum_{F \in \mathcal F_{j+\ell}(P)} \P((F-x) \cap L \neq \varnothing ).
$$
Since $F-x$ is convex and does not contain $0$,
$$
(F-x) \cap L \neq \varnothing
\quad \Longleftrightarrow \quad
L \cap \pos(F-x) \neq \{0\}.
$$
Moreover, every face $F \in \mathcal F_{j+\ell}(P)$ is a simplex. Since $x\notin \aff(F)$, $\pos(F-x)$ is a simplicial cone of dimension $j+\ell+1$. Hence, the conic Crofton formula~\eqref{eq:conic_crofton} yields

\begin{align}
	\E f_j(P\cap(x+L))
	&=
	2 \sum_{F \in \mathcal F_{j+\ell}(P)} \sum_{q=0}^{\lfloor j/2\rfloor} \upsilon_{\ell+1+2q} (\pos(F-x)).
	\label{eq:expected_faces_general_polytope}
\end{align}


\medskip
\noindent
\emph{Proof of (a).} Suppose $P = \conv(y_0, \ldots, y_n)$, where $y_0, \ldots, y_n$ satisfy~\eqref{eq:global_orthocentric_simplex}. The $(j+\ell)$-dimensional faces of $P$ are
precisely
$$
F_I := \conv(y_i:i\in I),
\qquad
I \subseteq \{0,\ldots,n\},
\qquad
|I| = j+\ell+1.
$$
Thus,~\eqref{eq:expected_faces_general_polytope} immediately gives~\eqref{eq:expected_faces_orthocentric_simplex}.

It remains to verify that Proposition~\ref{prop:intrinsic_volumes_orthocentric_cone} applies with the parameters in~\eqref{eq:parameters_orthocentric_face}.
Fix such a subset $I$.
Since $F_I$ is a proper face of $P$, the assumption on $x$ gives $x \notin \aff(y_i: i \in I)$. Consequently, the vectors $(y_i-x)_{i\in I}$ are linearly independent.

Note that
$
\pos(y_i-x:i\in I)
=
\pos(\tau_i(y_i-x): i\in I).
$
For $i,j\in I$, using~\eqref{eq:global_orthocentric_simplex}, we obtain
\begin{align*}
\lan \tau_i(y_i-x), \tau_j(y_j-x) \ran
&=
\tau_i \tau_j \left( \lan  y_i,y_j \ran - \lan y_i,x \ran - \lan y_j ,x \ran +\|x\|^2 \right)\\
&=
\delta_{ij} - \tau_i \tau_j \left( \lan y_i,x \ran + \lan y_j,x \ran + \frac{1}{s} - \|x\|^2 \right)
=
\delta_{ij} - t_i r_j - r_i t_j,
\end{align*}
with
$$
t=(\tau_i)_{i\in I}
\qquad \text{ and }
r
=
\left( \tau_i \left( \lan y_i,x\ran + \frac{1}{2s} -\frac{\|x\|^2}{2} \right) \right)_{i\in I}.
$$
Hence the Gram matrix of $(\tau_i(y_i-x))_{i\in I}$ is
\begin{align*}
\left(\lan \tau_i(y_i-x),\tau_j(y_j-x) \ran \right)_{i,j\in I}
&=
\Id_{|I|}-tr^T-rt^T.
\end{align*}
Proposition~\ref{prop:intrinsic_volumes_orthocentric_cone} therefore applies with precisely the parameters in~\eqref{eq:parameters_orthocentric_face}.

\medskip
\noindent
\emph{Proof of (b).}
Suppose
$
P
=
\conv (\tau_1^+ e_1, \ldots, \tau_n^+ e_n, -\tau_1^- e_1, \ldots, -\tau_n^- e_n).
$
The $(j+\ell)$-dimensional faces of $P$ are precisely the simplices
$$
F_{I,\eps}
=
\conv \left( \eps_i \tau_i^{[\eps_i]} e_i: i\in I \right),
\qquad
|I| = j+\ell+1,
\qquad
\eps \in \{-1, +1\}^I.
$$
Formula~\eqref{eq:expected_faces_asymmetric_crosspolytope} now follows directly from~\eqref{eq:expected_faces_general_polytope}.

It remains to verify that Proposition~\ref{prop:intrinsic_volumes_orthocentric_cone} applies with the parameters in~\eqref{eq:parameters_crosspolytope_face}. Fix such $I$ and $\eps$. Since $F_{I,\varepsilon}$ is a proper face of $P$, the assumption on $x$ gives
$$
x \notin \aff\left( \eps_i \tau_i^{[\eps_i]} e_i: i\in I \right).
$$

Hence, the vectors
$
\left( \eps_i \tau_i^{[\eps_i]} e_i -x \right)_{i\in I}
$
are linearly independent.

Note that
$$
\pos\left( \eps_i \tau_i^{[\eps_i]} e_i - x: i\in I \right)
=
\pos\left( \frac{1}{\tau_i^{[\eps_i]}} \left( \eps_i \tau_i^{[\eps_i]} e_i - x \right) :i\in I\right).
$$
As in part~(a), a direct computation shows that the Gram matrix of the latter generators is
$$
\Bigg( \bigg\lan
\frac{1}{\tau_i^{[\eps_i]}} \left( \eps_i \tau_i^{[\eps_i]} e_i - x \right)
,
\frac{1}{\tau_j^{[\eps_j]}} \left( \eps_j \tau_j^{[\eps_j]}  e_j - x \right)
\bigg\ran \Bigg)_{i,j\in I}
=
\Id_{|I|}-tr^T-rt^T,
$$
with
$$
t
=
\bigg(\frac{1}{\tau_i^{[\eps_i]}} \bigg)_{i\in I},
\qquad
r
=
\bigg(\eps_i x_i- \frac{\|x\|^2}{2\tau_i^{[\eps_i]}} \bigg)_{i\in I}.
$$
Proposition~\ref{prop:intrinsic_volumes_orthocentric_cone} therefore applies with precisely the parameters in~\eqref{eq:parameters_crosspolytope_face}.
\end{proof}

\begin{remark}[Random linear sections]
Taking $x=0$ in Theorem~\ref{theo:expected_faces_shifted_sections} gives formulas for random linear sections.
\begin{enumerate}
\item
In part~(a), the parameters in~\eqref{eq:parameters_orthocentric_face} satisfy $r=t/(2s)$. Hence, all solid angles entering the conic intrinsic volumes can be expressed as integrals of dimension at most one, using the linearly dependent cases of Proposition~\ref{prop:intrinsic_volumes_orthocentric_cone}.

If additionally $\tau_0=\ldots=\tau_n$, then $P$ is a regular
simplex centered at the origin, and
$$
\lan \tau_i y_i,\tau_k y_k\ran
=
\delta_{ik}-\frac1{n+1},
\qquad i,k\in\{0,\ldots,n\}.
$$
Thus, for $|I|=j+\ell+1$, the cones $\pos(y_i:i\in I)$
are isometric to those considered
in~\cite[Proposition~4.4]{KabluchkoSeidel2022}.
The outer sum in~\eqref{eq:expected_faces_orthocentric_simplex}
consists of $\binom{n+1}{j+\ell+1}$ equal terms,
so this formula reduces to the regular-simplex formula
in~\cite[Theorem~4.5]{KabluchkoSeidel2022}.

\item
In part~(b), $P$ is a crosspolytope. Here, the parameters in~\eqref{eq:parameters_crosspolytope_face} satisfy $r=0$. Since all $\tau_i^\pm$ are positive,
$$
\pos\left(\eps_i\tau_i^{[\eps_i]}e_i:i\in I\right)
=
\pos(\eps_i e_i:i\in I),
$$
which is a $(j+\ell+1)$-dimensional orthant. By~\eqref{eq:conic_intr_volume_general}, its $k$-th conic intrinsic volume is
$$
2^{-(j+\ell+1)}\binom{j+\ell+1}{k},
\qquad k\in\{0,\ldots,j+\ell+1\}.
$$
Hence,~\eqref{eq:expected_faces_asymmetric_crosspolytope} simplifies to
$$
\E f_j(P\cap L)
=
2 \binom{n}{j+\ell+1} \sum_{q=0}^{\lfloor j/2\rfloor} \binom{j+\ell+1}{\ell+1+2q}
=
2 \binom{n}{j+\ell+1} \sum_{h=\ell}^{j+\ell} \binom{j+\ell}{h},
$$
where the second equality follows by applying the identity
$$
\binom{j+\ell+1}{\ell+1+2q}
=
\binom{j+\ell}{\ell+2q}
+
\binom{j+\ell}{\ell+1+2q}
$$
to each summand. In particular, these expected face numbers are independent of the parameters $\tau_i^+,\tau_i^-$ and coincide with those of the regular crosspolytope given in~\cite[Theorem~4.5]{KabluchkoSeidel2022}.
\end{enumerate}
\end{remark}

\appendix
\section{Face-number identity for random sections}\label{sec:appendix}
In this appendix, we give a proof of the face-number identity~\eqref{eq:face_number_almost_sure_identity} for sections of polytopes. For a polytope $Q \subseteq \R^n$, let $\partial Q$ denote its boundary.

\begin{proposition}\label{prop:app_face_number_identity}
For $n\in \{2,3,\ldots\}$, let $Q \subseteq \R^n$ be a full-dimensional polytope such that $0 \notin \partial Q$, and let $\ell \in \{1, \ldots, n-1\}$. Let $L$ be distributed uniformly on the Grassmannian of $(n-\ell)$-dimensional subspaces of $\R^n$. Almost surely, for all $j\in \{0, \ldots, n-\ell - 1\}$,
\begin{equation}\label{eq:app_face_number_identity}
f_j(Q \cap L)
=
\sum_{G \in \mathcal F_{j+\ell} (Q)} \ind_{\{G \cap L \neq \varnothing\}}.
\end{equation}
\end{proposition}
Here, we use the convention $f_j(\varnothing) = 0$.
Under the stronger assumption \(0\in\operatorname{Int}(Q)\), this identity is~\cite[Proposition~4.1]{KabluchkoSeidel2022}.
We follow the strategy of their proof in \cite[Section~5.3]{KabluchkoSeidel2022}. Its first step, \cite[Proposition~5.5]{KabluchkoSeidel2022}, is to show that \(L\) is almost surely in general position with respect to the affine hulls of the faces of \(Q\). In particular, for every face \(G\), this implies
\[
L\cap\operatorname{aff}(G)=\varnothing
\qquad\text{if }\dim G<\ell,
\]
and
\[
L\cap \aff(G) \neq \varnothing
\text{ and }
\dim\bigl(L\cap \aff(G)\bigr)
=
\dim G-\ell
\qquad\text{if }\dim G\geq\ell.
\]

Under our weaker assumption $0\notin\partial Q$, the first property may fail at the level of affine hulls: if $0\in\operatorname{aff}(G)$, then $L\cap\operatorname{aff}(G)$ necessarily contains the origin. Nevertheless, we shall show that, almost surely, $L$ does not intersect the face $G$ itself whenever $\dim G<\ell$. The next lemma makes this precise and provides the replacement for \cite[Proposition~5.5]{KabluchkoSeidel2022} needed here.

\begin{lemma}\label{lem:app_general_position_assumption}
Let $Q \subseteq \R^n$ be an $n$-dimensional polytope such that $0 \notin \partial Q$, let $\ell \in \{1,\ldots,n-1\}$, and let $L$ be uniformly distributed on the Grassmannian of $(n-\ell)$-dimensional linear subspaces of $\R^n$. Then, almost surely, the following properties hold simultaneously for every face $G$ of $Q$.
\begin{enumerate}
\item If $\dim(G) < \ell$, then
$
L \cap G
=
\varnothing.
$

\item If $\dim(G) \geq \ell$, then
$
L \cap \aff(G)
\neq
\varnothing
$
and
$
\dim(L \cap \aff(G))
=
\dim(G) - \ell.
$
\end{enumerate}
\end{lemma}

\begin{proof}
For every fixed linear subspace \(W\subseteq\R^n\),~\cite[Lemma~13.2.1]{SW08} (or also~\cite[Lemma~2.1]{GoodeySchneider1980}) says that
\begin{equation}\label{eq:app_lemma_general_pos_intersection}
\dim(L\cap W)
=
\max(0, \dim(L) + \dim(W) - n)
=
\max(0, \dim(W) - \ell )
\qquad\text{almost surely}.
\end{equation}

Since $Q$ has only finitely many faces,  it suffices to fix a face $G$ of $Q$ and to show that the assertion applicable to $G$ holds almost surely. So let $G$ be a face of $Q$.

\medskip
\noindent
\emph{Part (a).} First suppose $\dim(G)<\ell$. Then,
$$
\dim(\lin G)
=
\left\{
\begin{aligned}
&\dim(G), && \text{ if } 0 \in \aff(G)\\
&\dim(G)+1, && \text{ if } 0 \notin \aff(G)
\end{aligned}
\right\}
\leq
\dim(G)+1
\leq
\ell.
$$
Applying~\eqref{eq:app_lemma_general_pos_intersection} with $W = \lin(G)$, it  follows that almost surely
$$
\dim(L \cap \lin(G))
=
\max(0, \dim(\lin G) - \ell)
=
0,
$$
and hence $L \cap \lin(G) = \{0\}$, and in particular $L \cap G \subseteq \{0\}$.
Since $\dim(G) < \ell \leq n-1$,  $G$ is a proper face of $Q$. Thus, $G \subseteq \partial Q$, so $0 \notin G$. Therefore, $L \cap G = \varnothing$ almost surely, which proves~(a).

\medskip
\noindent
\emph{Part (b).} Now suppose $\dim(G)\geq\ell$, and let $x_G\in G$. The set
$\aff(G)-x_G\subseteq\R^n$ is a linear subspace of the same dimension as $G$, so
\eqref{eq:app_lemma_general_pos_intersection}, applied with $W=\aff(G)-x_G$, gives
\begin{equation}\label{eq:app_dim_translated_face}
\dim\bigl(L\cap(\aff(G)-x_G)\bigr)
=
\max(0,\dim(G)-\ell)
=
\dim(G)-\ell
\qquad\text{almost surely}.
\end{equation}
From now on we restrict to the almost-sure event on which
\eqref{eq:app_dim_translated_face} holds. On this event,
$$
\dim\bigl(L+(\aff(G)-x_G)\bigr)
=
\dim(L)+\dim(G)-\dim\bigl(L\cap(\aff(G)-x_G)\bigr)
=
n,
$$
and hence $L+(\aff(G)-x_G)=\R^n$.

Let us now first show that $L \cap \aff(G)$ is non-empty. For this, note that  $L+(\aff(G)-x_G)=\R^n$ entails that $x_G$ can be written as $x_G=z+w$ for some $z\in L$ and $w \in \aff(G)-x_G$. Hence,
$$
z
=
x_G-w
\in
x_G + (\aff(G)-x_G)
=
\aff(G),
$$
so $z \in L \cap \aff(G)$, and therefore $L \cap \aff(G) \neq \varnothing$.
As the non-empty intersection of two affine subspaces, $L \cap \aff(G)$ is again an affine subspace of $\R^n$. It remains to determine its dimension.

With $z\in L\cap\aff(G)$ from above, we have
$$
\bigl(L\cap\aff(G)\bigr)-z
=
(L-z)\cap\bigl(\aff(G)-z\bigr)
=
L\cap\bigl(\aff(G)-x_G\bigr).
$$
Indeed, $L-z=L$ because $z\in L$, and $\aff(G)-z=\aff(G)-x_G$ because $z,x_G\in\aff(G)$.
Taking dimensions and using \eqref{eq:app_dim_translated_face}, we obtain
$$
\dim\bigl(L\cap\aff(G)\bigr)
=
\dim\bigl(L\cap(\aff(G)-x_G)\bigr)
=
\dim(G)-\ell .
$$
Thus $G$ satisfies (b) almost surely.
\end{proof}

As in~\cite[Proposition~5.6]{KabluchkoSeidel2022}, we isolate the deterministic part of the proof of Proposition~\ref{prop:app_face_number_identity} in the following three statements.

\begin{lemma}\label{lem:app_3_statements_for_face_identity}
Let $Q \subseteq\R^n$  be an $n$-dimensional polytope such that $0 \notin \partial Q$, let $\ell \in \{1,\ldots,n-1\}$, and let $L\subseteq \R^n$ be an $(n-\ell)$-dimensional linear subspace. 
Suppose that properties (a) and (b) of Lemma~\ref{lem:app_general_position_assumption} hold for every face
$G$ of $Q$.
Then, the following statements hold for every $j \in \{0,\ldots,n-\ell-1\}$.
\begin{enumerate}
\item[(1)] If $G_1,G_2 \in \mathcal F_{j+\ell}(Q)$ are distinct faces such that
$
G_1 \cap L \neq \varnothing
$
and
$
G_2 \cap L \neq \varnothing,
$
then
$$
G_1 \cap L
\neq
G_2\cap L.
$$

\item[(2)] If $G \in \mathcal F_{j+\ell}(Q)$ and $G \cap L \neq \varnothing$, then
$$
G \cap L\in \mathcal F_j(Q\cap L).
$$

\item[(3)] If $A\in \mathcal F_j(Q\cap L)$, then there exists $G \in \mathcal F_{j+\ell}(Q)$ such that
$$
A
=
G\cap L.
$$
\end{enumerate}
\end{lemma}

\begin{proof}[Proof of Proposition~\ref{prop:app_face_number_identity} assuming Lemma~\ref{lem:app_3_statements_for_face_identity}]

By Lemma~\ref{lem:app_general_position_assumption}, almost surely properties (a) and (b) of that lemma hold simultaneously for every face of $Q$. Fix a subspace $L$ with these properties. For every $j \in \{0,\ldots,n-\ell-1\}$, statements (1), (2), and (3) of Lemma~\ref{lem:app_3_statements_for_face_identity} show that the map
$$
\left\{ G\in \mathcal F_{j+\ell}(Q): G \cap L\neq\varnothing \right\}
\longrightarrow
\mathcal F_j(Q\cap L),
\qquad
G\mapsto G\cap L
$$
is well-defined (by (2)), injective (by (1)), and surjective (by (3)),
hence a bijection. 
Thus,
$$
f_j(Q\cap L)
=
\#\left\{ G \in \mathcal F_{j+\ell}(Q): G\cap L \neq \varnothing \right\}
=
\sum_{G \in \mathcal F_{j+\ell}(Q)} \ind_{\{G\cap L \neq \varnothing\}},
$$
which proves~\eqref{eq:app_face_number_identity}.
\end{proof}

A key step in the proof of~\cite[Proposition~5.6]{KabluchkoSeidel2022} is~\cite[Lemma~5.7]{KabluchkoSeidel2022}. Applied to a face inside its affine hull, that lemma shows that a face intersected by $L$ is also intersected in its relative interior. Under the modified general-position assumptions used here, the corresponding statement is the following.

\begin{lemma}\label{lem:app_relative_interior_intersection}
Let $Q$, $\ell$, and $L$ satisfy the assumptions of Lemma~\ref{lem:app_3_statements_for_face_identity}. Then, for every face $G$ of $Q$,
\begin{equation}\label{eq:app_relative_interior_intersection}
G \cap L
\neq
\varnothing
\quad\Longrightarrow\quad
L\cap\relint(G)
\neq
\varnothing.
\end{equation}
\end{lemma}
\begin{proof}
Fix a face $G$ of $Q$ such that $G \cap L \neq \varnothing$. Choose $x\in G \cap L$. Since $G$ is the disjoint union of the relative interiors of its faces, there exists a unique face $F$ of $G$ (hence $F$ is also a face of $Q$) such that $x \in \relint(F)$.

Since $F$ is a face of $Q$ and $L \cap F \neq \varnothing$, property~(a) of Lemma~\ref{lem:app_general_position_assumption}, which holds by assumption, implies $\dim(F) \geq \ell$.

If $G=F$, then $x\in L \cap \relint(G)$, so~\eqref{eq:app_relative_interior_intersection} holds. We may therefore assume that $F\subsetneqq G$.

In order to be able to construct a point in $L \cap \relint(G)$, we first show that $L+(\aff(F)-x)=\R^n$. Indeed, since $x\in L\cap\aff(F)$, we have
$
L\cap\aff(F)
=
x + (L \cap (\aff(F)-x)).
$
Property (b) of Lemma~\ref{lem:app_general_position_assumption}, which holds by assumption, therefore gives
$$
\dim(L \cap (\aff(F)-x))
=
\dim(L \cap \aff(F))
=
\dim F-\ell.
$$
Consequently,
\begin{align*}
\dim (L +  (\aff(F)-x) )
=
\dim(L) + \dim(F) -\dim(L \cap (\aff(F)-x) )
=
n,
\end{align*}
and hence indeed
$$
L+(\aff(F)-x)
=
\R^n.
$$

With this, we can now construct a point in $L \cap \relint(G)$, see also Figure~\ref{fig:app_proof_lemma}.

Choose $y \in \relint(G)$. By the preceding equality, there exist $v\in L$ and $u \in \aff(F)-x$ such that $y-x = v+u$.

Since $x \in \relint(F)$, there exists $\eps>0$ such that
$$
x+\left\{ z\in \aff(F)-x: \|z\|<\eps \right\}
=
\left\{w\in \R^n: \|w-x\| <\eps \right\} \cap \aff(F)
\subseteq
\relint(F).
$$
Since $u \in \aff(F)-x$, it follows that for all sufficiently small \(t\in(0,1)\),
$$
x - \frac{t}{1-t}u \in \relint(F)
\subseteq
G.
$$
For such $t$, we have
\begin{align*}
x+tv
=
x+t(y-x-u)
=
(1-t)\left(x-\frac{t}{1-t}u\right)+ty.
\end{align*}
Since $x - \frac{t}{1-t}u \in G$, $y\in\relint(G)$, and $t\in(0,1)$, the last expression belongs to $\relint(G)$. Here we used the fact that the open segment joining a point of a convex set to a point of its relative interior is contained in its relative interior. Hence,
$
x+tv\in\relint(G).
$
Moreover, since $x,v\in L$ and $L$ is a linear subspace, we have
$x+tv\in L$. Consequently,
$$
x+tv\in L\cap\relint(G),
$$
and thus $L\cap\relint(G)\neq\varnothing$. This proves~\eqref{eq:app_relative_interior_intersection}.

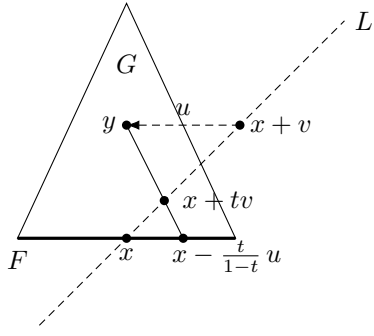
\begin{figure}[ht]
\centering
\begin{tikzpicture}[
    scale=1.15,
    >=Latex,
    every node/.style={font=\small}
]
\coordinate (A)  at (0.75,2);
\coordinate (B)  at (3.25,2);
\coordinate (C)  at (2,4.7);
\coordinate (X)  at (2,2);
\coordinate (Y)  at (2,3.3);
\coordinate (XV) at (3.3,3.3);
\coordinate (P)  at (2.65,2);
\coordinate (Q)  at (2.433,2.433);

\draw (A)--(B)--(C)--cycle;
    
\draw[very thick] (A)--(B);
\node[below=1pt of A] {$F$};
\node at (2,4) {$G$};

\draw[black,densely dashed]
    (1.,1.)--(4.5,4.5)
    node[right] {$L$};

\draw[black] (P)--(Y);


\draw[black, densely dashed,->]
    (XV)--(Y)
    node[midway,above] {$u$};


\fill (X) circle (1.5pt)
    node[below] {$x$};

\fill (Y) circle (1.5pt)
    node[left] {$y$};

\fill (XV) circle (1.5pt)
    node[right] {$x+v$};


\fill (P) circle (1.5pt)
    node[anchor=north west,xshift=-8pt,yshift=3pt]
    {$x-\frac{t}{1-t}\,u$};

\fill (Q) circle (1.5pt)
    node[right=3pt] {$x+tv$};

\end{tikzpicture}

\caption{The construction in the proof of Lemma~\ref{lem:app_relative_interior_intersection}.}
\label{fig:app_proof_lemma}
\end{figure}
\end{proof}

We are now in a position to prove Lemma~\ref{lem:app_3_statements_for_face_identity}, and thereby complete the proof of Proposition~\ref{prop:app_face_number_identity}.
The proofs of (1) and (2) follow the corresponding arguments in~\cite[proof of Proposition~5.6]{KabluchkoSeidel2022}. For (3), we give an alternative argument based on relative interiors.

\begin{proof}[Proof of Lemma~\ref{lem:app_3_statements_for_face_identity}]
\emph{Proof of (1).} Let $G_1, G_2$ be as in the statement. By Lemma~\ref{lem:app_relative_interior_intersection}, there exists $x\in L \cap \relint(G_1)$. Since  $G_1$ and $G_2$ are distinct faces of equal dimension, $\relint(G_1)$ and $G_2$ are disjoint. Since $x\in \relint(G_1)$, we have $x\notin G_2$. Thus, we have $x\in L \cap G_1$ and $x\notin L \cap G_2$, so $L \cap G_1 \neq L \cap G_2$.

\medskip
\noindent
\emph{Proof of (2).}
Let $G\in\mathcal F_{j+\ell}(Q)$ satisfy $G\cap L\neq\varnothing$.
First, $G\cap L$ is a face of $Q\cap L$. Indeed, suppose that
$z_1,z_2\in Q\cap L$ and $\lambda\in(0,1)$ satisfy
$$
(1-\lambda)z_1+\lambda z_2\in G\cap L.
$$
Since $G$ is a face of $Q$, it follows that $z_1,z_2\in G$. Hence, $z_1,z_2\in G\cap L$, which proves, by definition, that $G\cap L$ is a face of $Q\cap L$.

It remains to determine its dimension. By
Lemma~\ref{lem:app_relative_interior_intersection}, we may choose
$x\in L\cap\relint(G)$. Since $\dim(G)=j+\ell\geq\ell$, property~(b)
of Lemma~\ref{lem:app_general_position_assumption} gives
$$
\dim(L\cap \aff(G))
=
\dim(G) - \ell
=
j.
$$
Moreover, since $\relint(G)$ is relatively open in $\aff(G)$ and
$x\in\relint(G)$, there exists $\eps>0$ such that
$$
\left\{ z \in \aff(G): \|z-x\|<\eps \right\}
\subseteq
\relint(G)
\subseteq
G.
$$
Intersecting with $L$, we obtain
$$
\left\{ z\in L \cap \aff(G): \|z-x\| < \eps\right\}
\subseteq
G\cap L
\subseteq
L \cap \aff(G).
$$
The left-hand side is a relatively open neighborhood of $x$ in the $j$-dimensional affine space $L\cap\aff(G)$ and therefore has dimension $j$. The right-hand side also has dimension $j$. It follows that $\dim(G \cap L)=j$. 
Therefore, $ G\cap L\in\mathcal F_j(Q \cap L)$, which proves (2).

\medskip
\noindent
\emph{Proof of (3).}
Let $A\in\mathcal F_j(Q\cap L)$ and choose $x\in\relint(A)$. Since a polytope is the disjoint union of the relative interiors of its faces, there is a unique face $G$ of $Q$ such that $x\in\relint(G)$. Since $x\in G\cap L$, property~(a) of Lemma~\ref{lem:app_general_position_assumption} gives $\dim(G)\geq\ell$.

As in the proof of (2), $G\cap L$ is a face of $Q\cap L$. Hence, both $A$ and $G \cap L$ are faces of $Q \cap L$, so they are either equal, or their relative interiors are disjoint.
We chose $x\in \relint(A)$. Now we show that $x \in \relint(G\cap L)$ as well.

Indeed, since  $x \in \relint(G)$, as in the proof of (2), there exists $\eps>0$ such that
\[
\left\{
z\in L\cap\aff(G):\|z-x\|<\eps
\right\}
\subseteq
G\cap L
\subseteq
L\cap\aff(G).
\]
The set on the left-hand side is a non-empty relatively open subset of
$L\cap\aff(G)$ and therefore has affine hull $L\cap\aff(G)$. Taking
affine hulls in the displayed inclusions hence gives
$
\aff(G\cap L)=L\cap\aff(G).
$
Consequently, the first inclusion can be written as
\[
\left\{
z\in\aff(G\cap L):\|z-x\|<\varepsilon
\right\}
\subseteq
G\cap L,
\]
so indeed $x\in\relint(G\cap L)$.

Thus, $x \in \relint(A) \cap \relint(G\cap L)$. Since $A$ and $G\cap L$ are faces of $Q\cap L$ and distinct faces have disjoint relative interiors, it follows that $A=G\cap L$.

It remains to determine the dimension of $G$. Since $\dim(G)\geq\ell$, property~(b) of Lemma~\ref{lem:app_general_position_assumption} and the equality $\aff(G\cap L)=L\cap\aff(G)$ give
\[
j
=
\dim(A)
=
\dim(G\cap L)
=
\dim\bigl(L\cap\aff(G)\bigr)
=
\dim(G)-\ell.
\]
Consequently, $\dim(G)=j+\ell$, and hence $G \in\mathcal F_{j+\ell}(Q)$ and $A=G\cap L$. This proves~(3).
\end{proof}

\section*{Acknowledgement}
The author thanks Zakhar Kabluchko for helpful discussions and valuable comments on an earlier version of the manuscript.

Supported by the DFG under Germany's Excellence Strategy EXC 2044/2 - 390685587, Mathematics M\"unster: \emph{Dynamics-Geometry-Structure}, by the DFG priority program SPP 2265 \emph{Random Geometric Systems}, and by the DFG Research Training Group \emph{Rigorous Analysis of
Complex Random Systems} (RTG 3027, Project Number 524444762).

\section*{Declarations}

\subsection*{Statement on the use of generative AI}
ChatGPT assisted with mathematical discussions, the development and checking of arguments, literature searches and language editing. In particular, it suggested applying the Hubbard--Stratonovich identity~\eqref{eq:Hubbard_Stratonovich} to the Gaussian integral representation~\eqref{eq:solid_angle_simplicial_cone_general_explicit} to separate the integration variables.
Discussions with ChatGPT also contributed to the identification of the framework of simplicial cones admitting generators whose Gram matrix is the sum of a positive diagonal matrix and a low-rank symmetric matrix.

All suggestions produced by the AI were critically examined and
independently verified by the author.

\subsection*{Conflict of interest statement}
The author declares that he has no conflicts of interest.

\subsection*{Data availability statement}
We do not analyse or generate any datasets.

\addcontentsline{toc}{section}{References}
\bibliography{solid_angles_random_shifted_sections_bib}
\bibliographystyle{plainnat}

\vspace{1cm}

\footnotesize

\textsc{Philipp Schange: Institut f\"ur Mathematische Stochastik, Universit\"at M\"unster,
	Orl\'eans-Ring 10,
	48149 M\"unster, Germany}\\
\textit{E-mail}: \texttt{philipp.schange@uni-muenster.de}
\end{document}